\documentclass{article}
\usepackage[affil-it]{authblk}
\usepackage{amssymb, amsmath, dsfont, amsthm}
\usepackage{graphicx}
\usepackage{url} 
\usepackage{comment}
\usepackage{color}

\newcommand{\nl}{\operatorname{NumbSeq}}

\newcommand{\weight}{\operatorname{\Omega}}
\newcommand{\hp}{\operatorname{hg}}
\newcommand{\HP}{\operatorname{HG}}
\newcommand{\NMHP}{\operatorname{SHG}}
\newcommand{\SHP}{\operatorname{THG}}

\newcommand{\nmhp}{\operatorname{shg}}
\newcommand{\shp}{\operatorname{thg}}
\newcommand{\chp}{\operatorname{chg}}
\newcommand{\hpUV}{\operatorname{hg}^{\setminus U, V}}

\newcommand{\crit}{\Lambda}

\newcommand{\bigO}{\mathcal{O}}

\newcommand{\transp}[1]{#1^T}

\newtheorem{theorem}{Theorem}
\newtheorem{lemma}[theorem]{Lemma}
\newtheorem{corollary}[theorem]{Corollary}

\begin{document}

\title{Phase Transition of Random Non-Uniform Hypergraphs
\footnote{This work was partially founded by the ANR Boole, the ANR Magnum, the Amadeus program and the Univ Paris Diderot, Sorbonne Paris Cit\'e (UMR 7089).}}

\author{\'Elie de Panafieu} 
\affil{RISC Institute,\\
  J. Kepler University, Linz, Austria}

\maketitle

\begin{abstract}
Non-uniform hypergraphs appear in various domains of computer science 
as in the satisfiability problems and in data analysis.
We analyze a general model
where the probability for an edge of size~$t$
to belong to the hypergraph depends on a parameter~$\omega_t$
of the model.
It is a natural generalization of the models of graphs
used by Flajolet, Knuth and Pittel~\cite{FKP89} 
and Janson, Knuth, \L{}uczak and Pittel~\cite{JKLP93}.
The present paper follows the same general approach
based on analytic combinatorics.
We show that many analytic tools developed 
for the analysis of graphs can be extended surprisingly well
to non-uniform hypergraphs.
Specifically, we investigate random hypergraphs with a large number of vertices~$n$
and a complexity, defined as the \emph{excess}, proportional to~$n$.
We analyze their typical structure before, near and after the birth of the \emph{complex} components,
that are the connected components with more than one cycle.
Finally, we compute statistics of the model to link number of edges and excess.
\end{abstract}


    \section{Introduction}

In the seminal article~\cite{ER60},
Erd\H{o}s and R\'enyi discovered an abrupt change of the structure of a random graph
when the number of edges reaches half the number of vertices.
It corresponds to the emergence of the first connected component 
with more than one cycle, immediately followed by components
with even more cycles.
The combinatorial analysis of those components 
improves the understanding of the objects
modeled by graphs and has application
in the analysis and the conception of graph algorithm.
The same motivation holds for hypergraphs
which are used, among others, to represent databases and xor-formulas.

Much of the literature on hypergraphs is restricted to the uniform case,
where all the edges contain the same number of vertices.
In particular, the analysis of the birth of the complex component 
in terms of the size of the components and the order of the phase transition
can be found in~\cite{KL02}, \cite{Co04}, \cite{DM08}, \cite{GZVCN09} and~\cite{R10}.

There is no canonical choice for the size of a random edge
in a hypergraph; thus several models have been proposed.
One is developed in~\cite{SS85}, where the size of the largest connected component
is obtained using probabilistic methods.
It is our opinion that to be general,
a non-uniform hypergraph model needs one parameter
for each possible size of edges,
in order to quantify how often those edges appear.
In~\cite{DN04}, Darling and Norris define such a model,
the Poisson random hypergraphs model,
and analyze its structure via fluid limits 
of pure jump-type Markov processes.

We have not found in the literature much use of the generating function
of non-uniform hypergraphs to investigate their structure,
and we intend to fill this gap.
However, similar generating functions have been derived in~\cite{GK05}
for a different purpose:
Gessel and Kalikow use it to give a combinatorial interpretation 
for a functional equation of Bouwkamp and de Bruijn.
The underlying hypergraph model is a natural generalization
of the \emph{multigraph process}.

In Section~\ref{sec:definitions} we introduce the hypergraph models,
the probability distribution and the corresponding generating functions.
The important notion of \emph{excess} is also defined. 
Section~\ref{sec:hpnk} is dedicated to the asymptotic number
of hypergraphs with~$n$ vertices and excess~$k$.
Statistics on the random hypergraphs are derived,
including the limit distribution of the number of edges.
Section~\ref{sec:forest} focuses on hypergraphs
with small excess (subcritical), which are composed only 
of trees and unicycle components with high probability.
The critical excess at which the first complex component appears
is obtained in Section~\ref{sec:birth}. 
For a range of excess near and before this critical value,
we compute the probability that a random hypergraph 
contains no complex component. 
The classical notion of \emph{kernel} is introduced for hypergraphs
in Section~\ref{sec:kernels}.
It is then used to derive
the asymptotics of connected hypergraphs 
with~$n$ vertices and fixed excess~$k$.
%
We derive in Section~\ref{sec:hpcomplex} 
the structure of random hypergraphs in the critical window,
and obtain a surprising result: 
although the critical excess is generally different for graphs and hypergraphs,
both models share the same structure distribution 
exactly at their respective critical excess.
Finally, we give an intuitive explanation of the birth of the \emph{giant component}
in Section~\ref{sec:giant} and prove that there is with high probability
a component with an unbounded excess in random hypergraphs with supercritical excess.

    \section{Presentation of the Model} \label{sec:definitions}

In this paper, a hypergraph~$G$ is a multiset~$E(G)$ of~$m(G)$ edges.
Each edge~$e$ is a multiset of~$|e|$ vertices in~$V(G)$, where~$|e| \geq 2$.
The vertices of the hypergraph are labelled from~$1$ to~$n(G)$.
We also set~$l(G)$ for the \emph{size} of~$G$, defined by
\[
  l(G) = \sum_{e \in E(G)} |e| = \sum_{v \in V(G)} \deg(v).
\]
Those notions are illustrated in figure~\ref{fig:example}.

\begin{figure}
\begin{center}
\includegraphics[scale = 0.33]{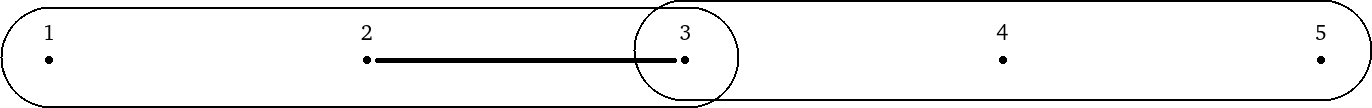}
\end{center}
\caption{Hypergraph with~$n=5$ vertices, $m = 3$ edges, excess~$k=0$, size~$l=8$ and~$\nl = 432$.
There is one cycle, which links the vertices $2$ and $3$.}
\label{fig:example}
\end{figure}

The notion of \emph{excess} was first used for graphs
in~\cite{W77}, then named in~\cite{JKLP93},
and finally extended to hypergraphs in~\cite{KL97}.
The excess of a connected component~$C$ is always greater than or equal to $-1$.
It expresses how far from a tree it is:
$C$ is a tree if and only if its excess is~$-1$,
contains exactly one cycle if its excess is~$0$,
and is said to be \emph{complex} if its excess is strictly positive.
Intuitively, a connected component with high excess is ``hard''
to treat for a backtracking algorithm.
The excess~$k(G)$ of a hypergraph~$G$ is defined by
\[
  k(G) = l(G) - n(G) - m(G).
\]

A hypergraph may contain several copies of the same edge
and a vertex may appear more than once in an edge;
thus we are considering multihypergraphs.
A hypergraph with no loop nor multiple edge is said to be \emph{simple}.
Since each edge is a multiset of vertices, 
the edges of a hypergraph~$G$ form a multiset of multisets of vertices~$E(G)$.
We define~$\nl(G)$ as the number of distinct orderings of the vertices in~$E(G)$.
For example, two possible orderings for the hypergraph from figure~\ref{fig:example} are
$((5,3,4),(3,2),(2,1,3))$ and $((2,3),(5,4,3),(2,1,3))$,
while $(5,1,4),(1,2),(2,3,1)$ would describe a different hypergraph.
In summary, $\nl(G)$ is the number of ways to write~$E(G)$ 
as a sequence of sequences of vertices.
If~$G$ is simple, then~$\nl(G)$ is equal to~$m(G)! \prod_{e \in E(G)} |e| !$,
otherwise it is smaller.
We associate to any family~$\mathcal{F}$ of hypergraphs the generating function
\begin{equation} \label{eq:gf}
  F(z,w,x) = 
    \sum_{G \in \mathcal{F}} 
      \frac{\nl(G)}{m(G)!} 
      \left( \prod_{e \in E(G)} \frac{\omega_{|e|}}{|e|!} \right)
      w^{m(G)} x^{l(G)} \frac{z^{n(G)}}{n(G)!}
\end{equation}
where~$\omega_t$ marks the edges of size~$t$, $w$ the edges, 
$x$ the size of the graph and~$z$ the vertices.
Therefore, we count hypergraphs with 
a \emph{weight}~$\kappa$
\begin{equation} \label{eq:weight}
  \kappa(G) = \frac{\nl(G)}{m(G)!} \prod_{e \in E(G)} \frac{\omega_{|e|}}{|e|!}
\end{equation}
that is the extension to hypergraphs of the \emph{compensation factor} 
defined in Section~$1$ of~\cite{JKLP93}.
We will conveniently refer to the sum of the weights of the hypergraphs in~$\mathcal{F}$
as the \emph{number of hypergraphs} in~$\mathcal{F}$.
If~$\mathcal{F}$ is a family of simple hypergraphs, then
this number of hypergraphs is the actual cardinality of $\mathcal{F}$.
In this case, we obtain the simpler and natural expression
\begin{equation} \label{eq:simplegf}
  F(z,w,x) = 
    \sum_{G \in \mathcal{F}} 
    \Big( \prod_{e \in E(G)} \omega_{|e|} \Big)
    w^{m(G)} x^{l(G)} \frac{z^{n(G)}}{n(G)!}.
\end{equation}
Observe that the generating function of the subfamily of hypergraphs of excess~$k$
is~$[y^k] F\left(z/y, w/y, x y\right)$,
where~$[x^n] \sum_k a_k x^k$ denotes the coefficient~$a_n$.

We define the exponential generating function of the edges as
\[
  \weight(z) := \sum_{t \geq 2} \omega_t \frac{z^t}{t!}.
\]
From now on, the~$(\omega_t)$ are considered as a bounded sequence 
of nonnegative real numbers with~$\omega_0 = \omega_1 = 0$.
The value~$\omega_t$ represents how likely an edge of size~$t$ is to appear.
Thus, for graphs we get~$\weight(z) = z^2/2$, 
for $d$-uniform hypergraphs (\textit{i.e.} with all edges of size $d$)
we have~$\weight(z) = z^d/d!$,
for hypergraphs with sizes of edges restricted to a set~$S$
we have~$\weight(z) = \sum_{s \in S} z^s$
and for hypergraphs with weight~$1$ for all size of edge~$\weight(z) = e^z-1-z$.
The hypothesis $\omega_0 = \omega_1 = 0$
means that the edges of size $0$ or $1$ are forbidden
(more specifically, any hypergraph that contains such an edge
will be counted with weight $0$).
To simplify the saddle point proofs,
we also assume that~$\weight(z)/z$ cannot
be written as~$f(z^d)$ for an integer~$d > 1$ and a power series~$f$
with a non-zero radius of convergence. 
This implies that~$e^{\weight(z)/z}$ is aperiodic.
Therefore, we do not treat
the important, but already studied, 
case of $d$-uniform hypergraphs for~$d > 2$
(those are the hypergraphs where all the edges have same size~$d$).

The generating function of all hypergraphs is
\begin{equation} \label{eq:hp}
  \hp(z,w,x) = \sum_n e^{w \weight(n x)} \frac{z^n}{n!}.
\end{equation}
This expression can be derived from~\eqref{eq:gf}
or using the symbolic method presented in~\cite{FS09}.
Indeed, $\weight(n x)$ represents an edge of size marked by~$x$
and~$n$ possible types of vertices,
and~$e^{w \weight(n x)}$ a set of edges.
For the family of simple hypergraphs,
\begin{equation} \label{eq:nmhp}
  \nmhp(z,w,x) = \sum_n \left( \prod_t (1 + \omega_t x^t w)^{\binom{n}{t}} \right) \frac{z^n}{n!}.
\end{equation}
Similar expressions have been derived in~\cite{GK05}.
The authors use them to give a combinatorial interpretation
of a functional equation of Bouwkamp and de Bruijn.

A hypergraph with~$n$ vertices and~$m$ labelled edges
can be represented by an~$(n,m)$-matrix~$M$ 
with nonnegative integer coefficients,
the coefficient~$M_{v,e}$ being the number of occurrences
of the vertex~$v$ in the edge~$e$.
In this representation, multigraphs correspond
to matrices where the sum of the coefficients
on each column is equal to~$2$.
Simple hypergraphs correspond to matrices
with~$\{0,1\}$ coefficients
that do not contain two identical columns.
Let us consider a hypergraph~$G$ and a matrix representation~$M$ of it.
A hypergraph~$H$ is said to be the \emph{dual} of~$G$
if the transpose~$\transp{M}$ of~$M$ represents it.
In other words, $H$ is obtained from~$G$
by reversing the roles of vertices and edges,
of degrees and sizes of the edges.
Therefore, the choice of weighting the edges depending of their size
can be transposed into weights on the vertices with respect to their degrees.
Figure~\ref{fig:dual} displays a dual of the hypergraph of figure~\ref{fig:example}.
This notion will be useful in the proof of Theorem~\ref{th:kernel}.

\begin{figure}
\begin{center}
\includegraphics[scale = 0.4]{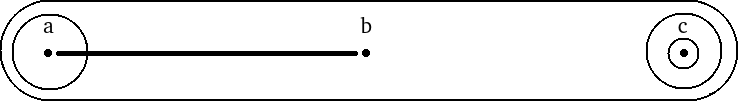}
\end{center}
\caption{One of the duals of the hypergraph of figure~\ref{fig:example}. 
The vertex~$c$, of degree~$3$, corresponds to the edge~$(3,4,5)$, of size~$3$, in figure~$\ref{fig:example}$.}
\label{fig:dual}
\end{figure}

Comparing~\eqref{eq:gf} with~\eqref{eq:simplegf},
simple hypergraphs may appear more natural
than hypergraphs.
But their generating function is more intricate,
their matrix representations satisfy more complex constraints
and the asymptotic results on hypergraphs can often 
be extended to simple hypergraphs.
Furthermore, experience has shown that multigraphs 
appear as often as simple graphs in applications.
This is why we do not confine our study to simple hypergraphs.

So far, we have adopted an enumerative approach of the model,
but there is a corresponding probabilistic description.
Let us define~$\HP_{n,k}$ (resp.~$\NMHP_{n,k}$) 
as the set of hypergraphs (resp. simple hypergraphs) 
with~$n$ vertices and excess~$k$,
equipped with the probability distribution induced by the weights~\eqref{eq:weight}.
Therefore, the hypergraph~$G$ occurs with probability~$\kappa(G)/\sum_{H \in \HP_{n,k}} \kappa(H)$.

    \section{Hypergraphs with~$n$ Vertices and Excess~$k$} \label{sec:hpnk}

In this section, we derive the asymptotic number of hypergraphs
and simple hypergraphs with~$n$ vertices and global excess~$k$.
This result is interesting by itself and is a first step
to find the excess~$k$ 
at which the first component with strictly positive excess
is likely to appear.
Statistics on the number of edges are also derived.

\begin{theorem} \label{th:hp}
  Let~$\lambda$ be a strictly positive real value and~$k = (\lambda - 1) n$,
  then the sum of the weights of the hypergraphs in~$\HP_{n,k}$ is
  \[
    \hp_{n,k} \sim 
      \frac{n^{n+k}}{\sqrt{2 \pi n}}
      \frac{e^{\frac{\weight(\zeta)}{\zeta} n}}{\zeta^{n+k}}
      \frac{1}{\sqrt{\zeta \weight''(\zeta) - \lambda}}
  \] 
  where~$\Psi(z)$ denotes the function~$\weight'(z) - \frac{\weight(z)}{z}$
  and~$\zeta$ is defined by~$\Psi(\zeta) = \lambda$.
  A similar result holds for simple hypergraphs:
  \[
    \nmhp_{n,k} \sim
      \frac{n^{n+k}}{\sqrt{2 \pi n}}
      \frac{e^{\frac{\weight(\zeta)}{\zeta} n}}{\zeta^{n+k}}
      \frac{\exp\left({-\frac{\omega_2^2 \zeta^2}{4} - \frac{\zeta \weight''(\zeta)}{2}}\right)}{\sqrt{\zeta \weight''(\zeta) - \lambda}}.
  \]
  More precisely, if~$k = (\lambda - 1) n + x n^{2/3}$
  where~$x$ is bounded,
  then the two previous asymptotics are multiplied by
  a factor~$\exp\left( \frac{-x^2}{2(\zeta \weight''(\zeta) - \lambda)} n^{1/3}  + \frac{x^3}{6} \frac{\zeta^2 \weight'''(\zeta) + \lambda}{(\zeta \weight''(\zeta) - \lambda)^3} \right)$.
\end{theorem}

\begin{proof}
  With the convention~\eqref{eq:gf},
  the sum of the weights of the hypergraphs 
  with~$n$ vertices and excess~$k$ is
  \begin{align*}
    n! [z^n y^k] \hp(z/y, 1/y, y)
    = n! [z^n y^k] \sum_n e^{\frac{\weight(n y)}{y}} \frac{(z/y)^n}{n!}
    = n^{n+k} [y^{n+k}] e^{\frac{\weight(y)}{y} n}.
  \end{align*}
  The asymptotics is then extracted using the Large Powers Scheme
  presented in~\cite[Chapter VIII]{FS09}.
  Observe that~$\Psi(z) = \sum_{t} \omega_t (t-1) \frac{z^{t-1}}{t!}$
  has nonnegative coefficients, so there is a unique solution of~$\Psi(\zeta) = \lambda$,
  and that~$\Psi(\zeta) = \lambda$
  implies~$\zeta \weight''(\zeta) - \lambda > 0$.

  For simple hypergraphs, 
  the coefficient we want to extract from~\eqref{eq:nmhp} is now
  \[
    [y^{n+k}] \prod_t (1 + \omega_t y^{t-1})^{\binom{n}{t}} 
  = 
    \frac{n^{n+k}}{2i\pi} \oint 
      \exp {\textstyle \left(\sum_t \binom{n}{t} \log\left(1 + \omega_t \left(\frac{y}{n}\right)^{t-1}\right) \right)}
      \frac{dy}{y^{n+k+1}}.
  \]
  The sum in the exponential can be rewritten
  \[
    \frac{\weight(y)}{y} n + 
    \sum_t \binom{n}{t} 
      \left( \log (1 + \omega_t \left(\frac{y}{n}\right)^{t-1} ) - \omega_t \left(\frac{y}{n}\right)^{t-1} \right)
      - \left( \frac{n^t}{t!} - \binom{n}{t} \right) \omega_t \left(\frac{y}{n}\right)^{t-1}
  \]
  which is~$\frac{\weight(y)}{y} n -\frac{\omega_2^2 y^2}{4} - \frac{y \weight''(y)}{2} + \bigO(1/n)$ 
  when~$y$ is bounded (we use here the hypothesis that~$\omega_0 = \omega_1 = 0$).
  In the saddle point method, $y$ is close to~$\zeta$, 
  which in our case is fixed with respect to~$n$.
  Therefore, 
  \[
    n! [z^n y^k] \nmhp \left(\frac{z}{y}, \frac{1}{y}, y \right)
    \sim
    \exp\left({-\frac{\omega_2^2 \zeta^2}{4} - \frac{\zeta \weight''(\zeta)}{2}}\right) \hp_{n,k}.
  \]

  The constraint~$k = (\lambda - 1) n + x n^{2/3}$ is equivalent
  to~$k = (\bar{\lambda} - 1) n$ 
  with~$\bar{\lambda} = \lambda + x n^{-1/3}$.
  Since~$x$ is bounded, so is~$\bar{\lambda}$
  and the first part of the theorem can be applied.
  Let us consider the solution~$\bar{\zeta}$
  of~$\Psi(\bar{\zeta}) = \bar{\lambda}$.
  With the help of maple, we find
  \[
    \frac{e^{n \frac{\weight(\bar{\zeta})}{\bar{\zeta}}}}{\bar{\zeta}^{n+k}}
    =
    \frac{e^{n \frac{\weight(\zeta)}{\zeta}}}{\zeta^{n+k}}    
    \exp \left(
      - \frac{x^2 n^{1/3}}{2 (\zeta \weight''(\zeta) - \lambda)}
      + \frac{x^3}{6} 
        \frac{\zeta \weight'''(\zeta) + \lambda}{(\zeta \weight''(\zeta) - \lambda)^3}
      + \bigO(n^{-1/3})
    \right).
  \]
\end{proof}

The factor~$\exp\big({-\frac{\omega_2^2 \zeta^2}{4} - \frac{\zeta \weight''(\zeta)}{2}}\big)$ 
is the asymptotic probability for a hypergraph in~$\HP_{n,k}$ to be simple.
For graphs, with~$\weight(z) = z^2/2$ and~$\lambda = 1/2$,
we obtain the same factor~$e^{-3/4}$ as in~\cite{JKLP93}.

We study the evolution of hypergraphs as their excess increases.
This choice of parameter may seem less natural than the number of edges,
but the excess turns out to be a better measure of the complexity of a hypergraph
than its number of edges.
Indeed, it seems natural to assume that a large edge carries 
more information than an edge of size $2$.
Furthermore, we can compute statistics on the number of edges
of hypergraphs with~$n$ edges and excess~$k$.

\begin{theorem}
  Let~$\Psi(z)$ and~$\zeta$ be defined as in Theorem~\ref{th:hp},
  and~$G$ be a random hypergraph in~$\HP_{n,k}$ or in~$\NMHP_{n,k}$
  with~$k = (\lambda - 1) n$,
  then the number~$m$ of edges of~$G$ admits a limit law
  that is Gaussian with parameters
  \begin{align*}
    \mathds{E}_{n,k}(m) &= \frac{\weight(\zeta)}{\zeta} n, \\
    \mathds{V}_{n,k}(m) &= \left( \frac{\weight(\zeta)}{\zeta} - \frac{\lambda^2}{\zeta \weight''(\zeta) - \lambda} \right) n.
  \end{align*}
  Reversely, the expectation and variance of the excess~$k$
  of a random hypergraph with~$n$ vertices and~$m$ edges are
  \begin{align*}
    \mathds{E}_{n,m}(k) &= n m \frac{\weight'(n)}{\weight(n)} - n - m,\\
    \mathds{V}_{n,m}(k) &= \frac{n m}{\weight(n)} \left( n \weight''(n) - n \frac{\weight'(n)^2}{\weight(n)} + \weight'(n) \right).
  \end{align*}
\end{theorem}

\begin{proof}
  By extraction from Equation~\eqref{eq:hp}, the generating functions of 
  the hypergraphs in $\HP_{n,k}$, in $\NMHP_{n,k}$
  and the generating function of hypergraphs with $n$ vertices and $m$ edges are
  \begin{align*}
    \hp_{n,k}(w) &= n^{n+k} [y^{n+k}] e^{w \frac{\weight(y)}{y} n},\\
    \nmhp_{n,k}(w) &= n^{n+k} [y^{n+k}] e^{w \frac{\weight(y)}{y} n} 
      e^{- \frac{y \weight''(y)}{2} w - \frac{\omega_2^2 y^2}{4} w^2 + \bigO(1/n)},\\
    \hp_{n,m}(y) &= \frac{\weight(n y)^m}{y^{n+m} m!},
  \end{align*}
  where~$w$ and~$y$ mark respectively the number of edges and the excess.
  The probability generating function corresponding to
  the distribution of the number of edges~$m$ in $\HP_{n,k}$, in $\NMHP_{n,k}$
  and to the distribution of the excess~$k$ in hypergraphs with $n$ vertices and $m$ edges
  are
  \begin{align*}
    p_{n,k}(w) &= \frac{\hp_{n,k}(w)}{\hp_{n,k}(1)} = \frac{[y^{n+k}] e^{w \frac{\Omega(y)}{y} n}}{[y^{n+k}] e^{\frac{\Omega(y)}{y} n}},
    \\
    \mathit{sp}_{n,k}(w) &= \frac{\nmhp_{n,k}(w)}{\nmhp_{n,k}(1)} = \frac{[y^{n+k}] e^{w \frac{\weight(y)}{y} n} 
          e^{- \frac{y \weight''(y)}{2} w - \frac{\omega_2^2 y^2}{4} w^2 + \bigO(1/n)}}{[y^{n+k}] e^{\frac{\weight(y)}{y} n} 
                e^{- \frac{y \weight''(y)}{2}- \frac{\omega_2^2 y^2}{4} + \bigO(1/n)}},
    \\
    p_{n,m}(y) &= \frac{\hp_{n,m}(y)}{\hp_{n,m}(1)} = \left( \frac{\Omega(ny)}{\Omega(n)} \right)^{m} \frac{1}{y^{n+m}}.
  \end{align*}

  The expected excess in a random hypergraphs with $n$ vertices and $m$ edges is $\mathds{E}_{n,m}(k) = p'_{n,m}(1)$,
  and its variance is $\mathds{V}_{n,m}(k) = p''_{n,m}(1) + p'_{n,m}(1) - p'_{n,m}(1)$.
  We therefore compute the first derivative of this probability generating function
  \[
    p'_{n,m}(y) =
    \left( n m \frac{ \Omega'(n y) }{ \Omega(n y) } - \frac{n+m}{y} \right)
    p_{n,m}(y),
  \]
  and thus the second derivative evaluated at $1$ is equal to
  \[
    p''_{n,m}(1) =
    n^2 m \left( \frac{\Omega''(n)}{\Omega(n)} - \left( \frac{\Omega'(n)}{\Omega(n)} \right)^2 \right) + n + m
    + \left( n m \frac{\Omega'(n)}{\Omega(n)} - \frac{n+m}{y} \right)^2.
  \]
  The values of the mean and variance of the excess follow.

  We now turn to the computation of the limit law of the number of edges $m$ 
  in a random hypergraph of $\HP_{n,k}$.
  First, using the \emph{Large Powers Theorem} \cite[Theorem VIII.8]{FS09}
  we obtain
  \[
    [y^{n+k}] e^{w \frac{\Omega(y)}{y} n} = 
    \frac{e^{w \frac{\Omega(\zeta_w)}{\zeta_w} n}}{
      \zeta_w^{n+k} 
      \sqrt{ 2 \pi n \left( \zeta_w \Omega''(\zeta_w) - \frac{\lambda}{w} \right)}
    }
    \left( 1 + \bigO(n^{-1/2}) \right)
  \]
  uniformly for $w$ in a neighborhood of $1$, where $\zeta_w$ is characterized by the relation
  \[
    w \Psi(\zeta_w) = \lambda.
  \]
  Uniformly for $s$ in a neighborhood of $0$, we find for the Laplace transform of the number of edges
  $p_{n,k}(e^s) = \mathds{E}_{n,k}(e^{s\, m})$ 
  \[
    p_{n,k}(e^s)
    =
    e^{n A(s) + B(s)}
    \left( 1 + \bigO \left( n^{-1/2} \right) \right)
  \]
  where
  \[
    A(s) = 
    \frac{\weight(\zeta)}{\zeta} s
    + \left(
      \frac{\weight(\zeta)}{\zeta}
      - \frac{\lambda^2}{\zeta \weight''(\zeta) - \lambda}
    \right) \frac{s^2}{2}
    + \bigO(s^3).
  \]
  To prove the normal limit distribution of the number of edges~$m$
  in~$\HP_{n,k}$, we then apply a lemma of Hwang~\cite{H98}
  that can also be found in~\cite[Lemma IX.1]{FS09}.
  The means is then $n A'(0)$ and the variance $n A''(0)$.
  The same result holds for simple hypergraphs.
\end{proof}

The variance $\mathds{V}_{n,k}(m)$ of the number of edges
in a random hypergraph of $\HP_{n,k}$ is zero only
for uniform hypergraphs.
Indeed, the number of edges is then characterized
by the number of vertices and the excess,
and the corresponding random variable is degenerated only in that case.

    \section{Subcritical Hypergraphs} \label{sec:forest}

We follow the conventions established by Berge~\cite{B85}:
a \emph{walk} of a hypergraph~$G$ 
is a sequence~$v_0, e_1, v_1, \ldots, v_{t-1}, e_t, v_t$
where for all~$i$, $v_i \in V(G)$, $e_i \in E(G)$ and~$\{v_{i-1}, v_i\} \subset e_i$.
A \emph{path} is a walk in which all~$v_i$ and~$e_i$ are distinct.
A walk is a cycle if all~$v_i$ and $e_i$ are distinct, except~$v_0 = v_t$.
Various examples of cycles are presented in Figure~\ref{fig:cycles}.
\begin{figure}
\begin{center}
\includegraphics[scale = 0.5]{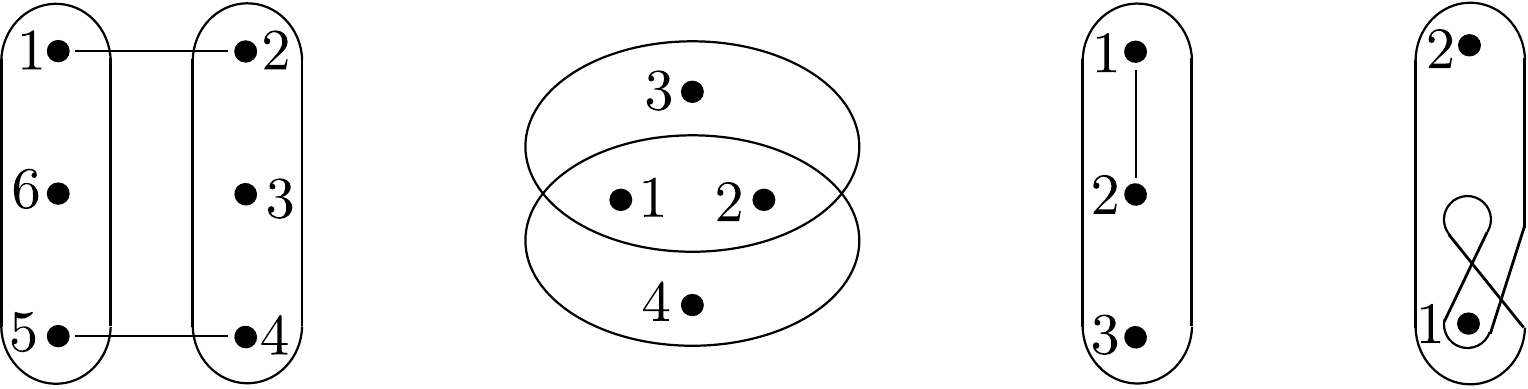}
\end{center}
\label{fig:cycles}
\caption{Each of those hypergraphs contains a cycle, respectively $(1,2,3,4,5,6)$, $(1,2)$, $(1,2)$ and $(1)$.
Observe that the last hypergraph is not simple.}
\end{figure}
Connectivity, trees and rooted trees are then defined in the usual way.

The generating function of edges is $\Omega(z)$.
We say that we replace a vertex with a \emph{hole}
when this vertex does not count in the size of the edge anymore.
The generating function of edges with one vertex replaced by a hole
is $\Omega'(z)$, because there are $t$ possible labels for the vertex removed
in an edge of size~$t$.
We can mark a vertex in an edge, 
and the corresponding generating function is $z \Omega'(z)$.
Holes allow us to clip edges together.
For example, two edges with one common vertex can be described
as an edge with a vertex replaced by a hole and an edge with a vertex marked.
The generating function of those hypergraphs is then $z \Omega'(z)^2/2$
(we divide by $2$ because the two edges have symmetrical roles).

A \emph{unicycle component} is a connected hypergraph 
that contains exactly one cycle.
We also define a \emph{path of trees} as 
a path with both ends replaced by holes and that contains no cycle,
plus a rooted tree hooked to each vertex
(except to the two ends of the path). 
It can equivalently be defined
as an unrooted tree with two distinct leaves replaced with holes.
The notion of path of trees is illustrated in Figure~\ref{fig:pathoftrees}.
\begin{figure}
\begin{center}
\includegraphics[scale = 0.5]{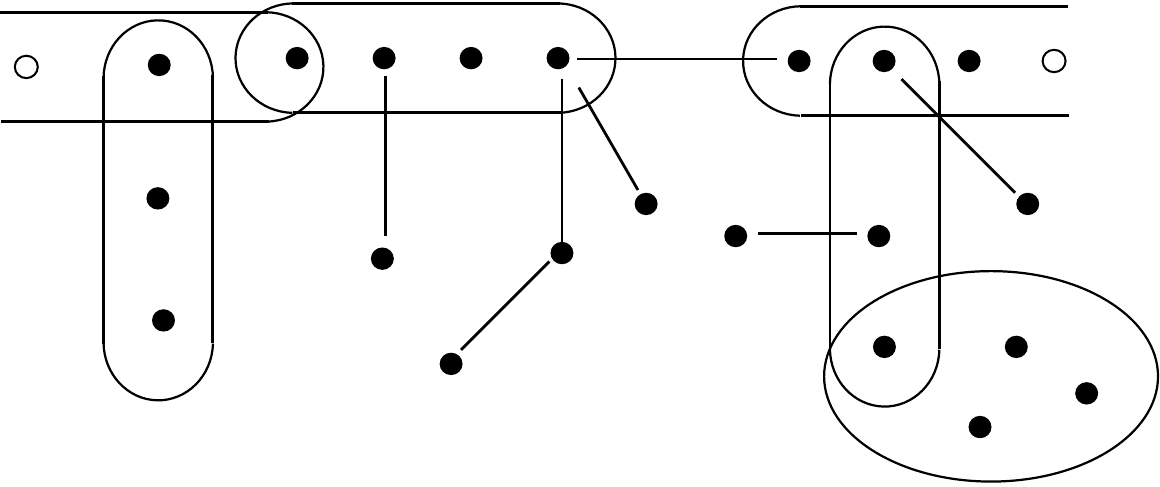}
\end{center}
\label{fig:pathoftrees}
\caption{A path of trees. The two white nodes represent holes. The labels have been omitted.}
\end{figure}

\begin{lemma} \label{th:TUV}
  Let~$T(z)$, $U(z)$, $V(z)$ and~$P(z)$ denote the generating functions 
  of rooted trees, unrooted trees, unicycle components 
  and paths of trees, using the variable~$z$
  to mark the number of vertices, then
  \begin{align}
    T(z) &= z e^{\weight'(T(z))}, \label{eq:T}\\
    U(z) &= T(z) + \weight(T(z)) - T(z) \weight'(T(z)), \label{eq:U}\\
    V(z) &= \frac{1}{2} \log \frac{1}{1 - T(z) \weight''(T(z))}, \label{eq:V}\\
    P(z) &= \frac{\weight''(T(z))}{1 - T(z) \weight''(T(z))}. \label{eq:P}
  \end{align}
\end{lemma}

\begin{proof}
  Those expressions can be derived using 
  the \emph{symbolic method} presented in~\cite{FS09}.
  The generating function of edges is~$\weight(z)$.
  If one vertex is marked, it becomes~$z \weight'(z)$
  and~$z \weight''(z)$ if another vertex is replaced by a hole.
  Equation~\eqref{eq:T} means that a rooted tree is 
  a vertex (the root) and a set of edges from which 
  a vertex has been replaced by a hole and the other vertices replaced by rooted trees.
  Equation~\eqref{eq:U} is a classical consequence of the dissymmetry theorem 
  described in~\cite{BLL97}. 
  Hypertrees have been studied using a combinatorial species approach by Oger in~\cite{O13}.
  It can be checked that~$z \partial_z U = T$,
  which, in a symbolic method, means that
  a tree with a vertex marked is a rooted tree.
  Unicycle components are cycles of rooted trees, which implies~\eqref{eq:V}.
\end{proof}

Combining the enumeration of hypergraphs
with the enumeration of forests,
we can investigate the birth of the first cycle
and the limit distribution of the number of cycles
in a hypergraph with small excess.

\begin{theorem} \label{th:forest}
  Let~$\Psi(z)$ denote the function~$\weight'(z) - \frac{\weight(z)}{z}$,
  $\tau$ be implicitly defined by~$\tau \weight''(\tau) = 1$
  and~$\crit = \Psi(\tau)$.
  Let us consider an excess~$k = (\lambda - 1) n$ where $0 < \lambda < \crit$
  and the value~$\zeta$ such that~$\Psi(\zeta) = \lambda$.
  With high probability, a hypergraph in~$\HP_{n,k}$ or~$\NMHP_{n,k}$
  contains no component with two cycles.
  The limit distribution of the number of cycles 
  of such a hypergraph follows a Poisson law of parameter 
  \[
    \frac{1}{2} \log \left( \frac{1}{1 - \zeta \weight''(\zeta)} \right)
  \]
  if the hypergraph is in~$\HP_{n,k}$, and
  \[
    \frac{1}{2} \log \left( \frac{1}{1 - \zeta \weight''(\zeta)} \right)
    - \frac{\omega_2^2 \zeta^2}{4}
    - \frac{\zeta \weight''(\zeta)}{2}
  \]
  if it is in~$\NMHP_{n,k}$.
\end{theorem}

\begin{proof}
  Let~$\SHP_{n,k}$ 
  denote the set of hypergraphs in~$\HP_{n,k}$
  that contains only trees and unicycle components.
  The excess of a tree is~$-1$, 
  the excess of a unicycle component is~$0$.
  Since the excess of a hypergraph is the sum of the excesses of its components, 
  each hypergraph in~$\SHP_{n,k}$ contains exactly~$-k$ trees.
  Let $F_{n,k}(u)$ denote the generating function of the number of cycles
  in hypergraphs of~$\SHP_{n,k}$, then
  \[
    F_{n,k}(u) = n! [z^n] \frac{U(z)^{-k}}{(-k)!} e^{u V(z)} 
  \]
  where~$u$ marks the cycles.
  In the Cauchy integral representation of the coefficient extraction in $F_{n,k}(u)$,
  we apply the change of variable $t = T(z)$ and obtain
  \[
    F_{n,k}(u) =
    \frac{n!}{(-k)!} \frac{1}{2i\pi} 
    \oint 
    \frac{t^{-k} (1- \Psi(t))^{-k} e^{\Omega'(t) n}}{(1 - t \Omega''(t))^{(u-1)/2}}
    \frac{d t}{t^{n+1}}.
  \]
  Since $k = (\lambda - 1) n$, this can be rewritten as a coefficient extraction
  \[
    F_{n,k}(u) = \frac{n!}{(-k)!} [t^{n+k}] 
    \frac{ \left( (1-\Psi(t))^{1-\lambda} e^{\Omega'(t)} \right)^n }{(1 - t \Omega''(t))^{(u-1)/2}}.
  \]
  We use the Large Powers Theorem~VIII.8 of~\cite{FS09}
  to extract the asymptotic.
  The saddle-point equation is
  \[
    t \frac{d}{d t} \log \left( (1-\Psi(t))^{1-\lambda} e^{\Omega'(t)} \right)
    = \frac{n+k}{n}
  \]
  and can be simplified into
  \[
    \Psi(t) (1-t \Omega''(t)) = \lambda (1-t\Omega''(t)).
  \]
  Its two roots are $\tau$ and $\zeta$, 
  where $\zeta$ is characterized by the relation 
  \[
    \Psi(\zeta) = \lambda.
  \]
  For $\lambda < \Lambda$, since $\Psi(\tau) = \Lambda$, we have $\zeta < \tau$,
  so $\zeta$ is the dominant saddle-point.
  Application of the theorem and Stirling approximations then lead to
  \[
    F_{n,k}(u)
    =
    \frac{n^{n+k}}{\sqrt{2 \pi n}}
    \frac{e^{\frac{\weight(\zeta)}{\zeta}n}}{\zeta^{n+k}}
    \frac{e^{\frac{u-1}{2} \log \left( \frac{1}{1 - T \weight''(T)} \right)}}{
    \sqrt{\zeta \weight''(\zeta) - \lambda} }
    ( 1 + \bigO(n^{-1/2}))
  \]
  uniformly for $u$ in a neighborhood of~$1$.
  Dividing by the cardinality of~$\HP_{n,k}$
  derived in Theorem~\ref{th:hp},
  we obtain the generating function
  of the limit probabilities of the number of cycles
  in~$\SHP_{n,k}$:
  \[
    \sum_t \mathds{P}(G \in \SHP_{n,k} \text{ and has~$t$ cycles}\ |\ G \in \HP_{n,k} ) u^t
    \sim
    e^{\frac{u-1}{2} \log \left( \frac{1}{1 - T \weight''(T)} \right)}.
  \]
  For~$u=1$, it is equal to~$1$,
  so with high probability,
  a hypergraph in~$\HP_{n,k}$ has no component
  with more than one cycle.
  For~$u = e^{i t}$, we recognize
  the characteristic function of
  a Poisson law with parameter~$\frac{1}{2} \log \left( \frac{1}{1 - T \weight''(T)} \right)$.

  The same computations hold
  for the analysis of simple hypergraphs, 
  except the generating function~$V(z)$ has to be replaced 
  by~$V(z) - \frac{T \weight''(T)}{2} - \frac{\omega_2^2 T^2}{4}$ 
  to avoid loops and multiple edges (in unicycle components, 
  those can only be two edges of size~$2$).
\end{proof}

More information on the length
of the first cycle and the size
of the component that contains it
could be extracted, following the approach of~\cite{FKP89}.

Observe that the value~$\Lambda$ defined in the theorem is always smaller than~$1$.
Indeed, the equalities
\[
	\tau \Omega''(\tau) = 1
	\quad \text{ and} \quad
	\Psi(\tau) = \Lambda
\]
are equivalent with
\[
	\sum_{t \geq 1}
	t (t-1) \omega_t \frac{\tau^{t-1}}{t!} =1
	\quad \text{ and} \quad
	\sum_{t \geq 1}
	(t-1) \omega_t \frac{\tau^{t-1}}{t!} = \Lambda,
\]
which implies~$\Lambda < 1$.

    \section{Birth of the complex components} \label{sec:birth}

Let us recall that a connected hypergraph is \emph{complex} 
if its excess is strictly positive.
In order to locate the global excess~$k$
at which the first complex component appears,
we compare the asymptotic numbers of hypergraphs
and hypergraphs with no complex component.

Theorem~\ref{th:shp} describes the limit probability
for a hypergraph not to contain any complex component.
A phase transition occurs when~$\frac{k}{n}$ 
reaches the critical value~$\crit -1$, defined in Theorem~\ref{th:forest}.
From an analytic point of view, 
this corresponds to the coalescence of two saddle points.
In this context, the Large Powers Scheme ceases to apply,
so we replace it with the following general theorem,
borrowed from~\cite{BFSS01} 
(see also Theorem~IX$.16$ of~\cite{FS09} for discussions 
and links with the \emph{stable laws} of probability theory)
and adapted for our purpose
(in the original theorem, $\mu = 0$). 
It is also close to Lemma~$3$ of~\cite{JKLP93}.

\begin{theorem} \label{th:coal}
  We consider a generating function~$H(z)$ with nonnegative coefficients
  and a unique isolated singularity at its radius of convergence~$\rho$.
  We also assume that it is continuable in~$\Delta := \{z\ |\ |z| < R, z \notin [\rho, R]\}$
  and there is a~$\lambda \in ]1;2[$
  such that~$H(z) = \sigma - h_1 (1 - z/\rho) + h_{\lambda} (1-z/\rho)^\lambda + \bigO((1-z/\rho)^2)$
  as~$z \rightarrow \rho$ in~$\Delta$. 
  Let~$k = \frac{\sigma}{h_1} n + x n^{1/\lambda}$ with~$x$ bounded, 
  then for any real constant~$\mu$
  \begin{equation} \label{eq:coal}
    [z^n] \frac{H^k(z)}{(1-z/\rho)^\mu} 
    \sim
    \sigma^k \rho^{-n} \frac{1}{n^{(1-\mu)/\lambda}} (h_1/h_\lambda)^{(1-\mu)/\lambda} 
    G\left(\lambda, \mu; \frac{h_1^{1+1/\lambda}}{\sigma h_\lambda^{1/\lambda}} x \right)
  \end{equation}
  where
  $
    G(\lambda,\mu; x) 
    =
    \frac{1}{\lambda \pi} \sum_{k \geq 0} \frac{(-x)^k}{k!} 
    \sin\left(\pi \frac{1-\mu +k}{\lambda} \right) \Gamma\left(\frac{1-\mu+k}{\lambda}\right).
  $
\end{theorem}

\begin{proof}
  In the Cauchy integral that represents~$[z^n] \frac{H^k(z)}{(1-z/\rho)^\mu}$
  we choose for the contour of integration a positively oriented loop, 
  made of two rays of angle $\pm \pi/(2\lambda)$
  that intersect on the real axis at~$\rho - n^{-1/\lambda}$,
  we set~$z = \rho(1 - t n^{-1/\lambda})$
  \[
    [z^n] \frac{H^k(z)}{(1-z/\rho)^\mu}
    \sim
    \frac{-\sigma^k \rho^{-n}}{2i\pi n^{(1-\mu)/\lambda}} 
    \int t^{-\mu} e^{\frac{h_\lambda}{h_1}t^\lambda} e^{-x \frac{h_1}{\sigma} t} dt
  \]
  The contour of integration comprises now two rays of angle~$\pm \pi/ \lambda$ intersecting at~$-1$.
  Setting~$u =  t^\lambda h_\lambda/h_1$, the contour transforms into a classical Hankel contour,
  starting from~$-\infty$ over the real axis, winding about the origin and returning to~$-\infty$.
  \[
    \frac{-\sigma^k \rho^{-n}}{2i\pi n^{(1-\mu)/\lambda}} \frac{1}{\lambda} (h_1/h_\lambda)^{(1-\mu)/\lambda} 
    \int_{-\infty}^{(0)} e^u e^{-x u^{1/\lambda} h_1^{1+1/\lambda} / (\sigma h_\lambda^{1/\lambda})}
      u^{\frac{1-\mu}{\lambda} - 1} du
  \] 
  Expanding the exponential, integrating termwise, and appealing to the complement
  formula for the Gamma function finally reduces this last form to~\eqref{eq:coal}.
\end{proof}

\begin{theorem} \label{th:shp}
  Let~$\Psi(z)$ denote the function~$\weight'(z) - \frac{\weight(z)}{z}$,
  $\tau$ be implicitly defined by~$\tau \weight''(\tau) = 1$,
  set $\crit = \Psi(\tau)$ and~$\gamma = 1+\tau^2 \weight'''(\tau)$.
  Let $G(\lambda,\mu;x)$ be defined as in Theorem~\ref{th:coal}.
  We consider an excess~$k = (\crit - 1) n + x n^{2/3}$ 
  where~$x$ is bounded.
  Then the number of the hypergraphs in~$\HP_{n,k}$
  with no complex component is equivalent to
  \begin{align} \label{eq:shpcrit}
      \frac{n^{n+k}}{\sqrt{2 \pi n}}
      \frac{e^{\frac{\weight(\tau)}{\tau} n}}{\tau^{n+k}}
      \frac{1}{\sqrt{1 - \crit}}
      \sqrt{\frac{3\pi}{2}} 
      e^{- \frac{x^2}{2(1-\crit)} n^{1/3} - \frac{x^3}{6(1-\crit)^2}} 
      G\left( \frac{3}{2}, \frac{1}{4}; - \frac{3^{2/3} \gamma^{1/3} x}{2 (1-\crit)}\right).
  \end{align}  
  For simple hypergraphs, this number is
  \[
    \frac{n^{n+k}}{\sqrt{2 \pi n}}
      \frac{e^{\frac{\weight(\tau)}{\tau} n}}{\tau^{n+k}}
      \frac{e^{-\frac{1}{2} - \frac{\omega_2^2 \tau^2}{4}}}{\sqrt{1 - \crit}}
      \sqrt{\frac{3\pi}{2}} 
      e^{- \frac{x^2}{2(1-\crit)} n^{1/3} - \frac{x^3}{6(1-\crit)^2}} 
      G\left( \frac{3}{2}, \frac{1}{4}; - \frac{3^{2/3} \gamma^{1/3} x}{2 (1-\crit)}\right).
  \]
\end{theorem}

\begin{proof}
  The number of hypergraphs in $\SHP_{n,k}$ has been derived in Theorem~\ref{th:forest}
  for $k = (\lambda - 1) n$ and $\lambda < \Lambda$.
  We now focus on the case $k = (\Lambda -1) n + x n^{2/3}$.
  The number of hypergraphs in $\SHP_{n,k}$ is again
  \[
    \shp_{n,k} = n! [z^n] \frac{U(z)^{-k}}{(-k)!} e^{V(z)}.
  \]
  The two saddle-points $\tau$ and $\zeta$ defined in Theorem~\ref{th:forest}
  now coalesce. Furthermore, $e^{V(z)}$ also has a singularity for $T(z) = \tau$.
  To extract the asymptotics, we apply Theorem~\ref{th:coal}.
  The Newton-Puiseux expansions of~$T$, $e^V$ and~$U$
  can be derived from Lemma~\ref{th:TUV}
  \begin{align*}
    T(z) &\sim \tau - \tau \sqrt{\frac{2}{\gamma}} \sqrt{1 - z/\rho},
  \\
    e^{V(z)} &\sim (2\gamma)^{-1/4} (1 - z/\rho)^{-1/4},
  \\
    U(z) &= \tau(1 - \Psi(\tau)) - \tau (1 - z/\rho) + 
      \tau \frac{2}{3} \sqrt{\frac{2}{\gamma}} (1 - z/\rho)^{3/2} + \bigO(1-z/\rho)^2,
  \end{align*}
  where~$\rho = \tau e^{- \weight'(\tau)}$.
  Using Theorem~\ref{th:coal}, we obtain
  \[
    \shp_{n,k} \sim
    \frac{n!}{(-k)!} \frac{\sqrt{3}}{2} \frac{(\tau(1-\crit))^{-k}}{\rho^n \sqrt{n}} 
    G\left( \frac{3}{2}, \frac{1}{4};- \frac{3^{2/3} \gamma^{1/3} x}{2 (1-\crit)}\right)
  \]
  which reduces to~\eqref{eq:shpcrit}.

  As in the proof of Theorem~\ref{th:forest},
  for the analysis of simple hypergraphs we replace 
  the generating function~$V(z)$
  with $V(z) - \frac{T \weight''(T)}{2} - \frac{\omega_2^2 T^2}{4}$.
\end{proof}

In Theorem~\ref{th:forest}, we have seen 
that when $k = (\lambda - 1)n$ with $\lambda < \Lambda$,
the probability for a random hypergraph in $\HP_{n,k}$ 
to contain only trees and unicyclic components tends to~$1$.
When~$k = (\crit - 1) n + \bigO(n^{1/3})$, this limit becomes~$\sqrt{2/3}$
because~$G(2/3, 1/4; 0)$ is equal to~$2/(3\sqrt{\pi})$.
It is remarkable that this value does not depend on~$\weight$, therefore
it is the same as in~\cite{FKP89} for graphs.
However, the evolution of this probability between the subcritical and the critical
ranges of excess depends on the parameters~$(\omega_t)$.

\begin{corollary} \label{th:probnocomplex}
  Let~$\tau$, $\crit$ and~$\gamma$ be defined as in Theorem~\ref{th:shp}.
  For~$k = (\lambda - 1) n$ and~$\lambda < \crit$,
  a hypergraph in~$\HP_{n,k}$ or in~$\NMHP_{n,k}$ has no complex component
  with high probability.
  For~$k = (\crit - 1) n + x n^{2/3}$ with~$x$ bounded,
  the limit probability that such a hypergraph has no complex component is
  \[
    \sqrt{\frac{3 \pi}{2}} \exp\left({\frac{-x^3 \gamma}{6(1 - \crit)^3}}\right)
    G\left( \frac{3}{2}, \frac{1}{4};- \frac{3^{2/3} \gamma^{1/3} x}{2 (1-\crit)}\right)
  \]
  where~$G$ is the function defined in Theorem~\ref{th:coal}.
\end{corollary}

\begin{proof}
  From the second assertion of Theorem~\ref{th:hp}
  we deduce the asymptotic number of hypergraphs in~$\HP_{n,k}$
  when~$k = (1 - \crit) n + x n^{2/3}$
  \[
    \hp_{n,k} \sim 
      \frac{n^{n+k}}{\sqrt{2 \pi n}}
      \frac{e^{\frac{\weight(\tau)}{\tau} n}}{\tau^{n+k}}
      \frac{e^{ 
      \frac{-x^2}{2(1 - \crit)} n^{1/3} 
      + \frac{x^3}{6} \frac{\gamma - 1 + \crit}{(1 - \crit)^3} 
    }}{\sqrt{1-\crit}}.
  \]
  Equation~\ref{eq:shpcrit} divided by this estimation of~$\hp_{n,k}$
  leads to the result announced. 
  The computations are the same for simple hypergraphs.
\end{proof}

Theorem~\ref{th:coal} does not apply when~$H(z)$ is periodic.
This is why we restricted $\weight(y)/y$ not to be of the form~$f(z^d)$
where~$d > 1$ and~$f(z)$ is a power series 
with a strictly positive radius of convergence.
An unfortunate consequence is that Theorems~\ref{th:hp} and~\ref{th:shp} do
not apply to the important but already analyzed case of $d$-uniform hypergraphs. 
However, the expression of the critical excess is still valid.
For the~$d$-uniform hypergraphs, $\weight(z) = \frac{z^d}{d!}$,
$\Psi(z) = \frac{(d-1)}{d!} z^{d-1}$
and~$\tau^{d-1} = (d-2)!$, 
so we obtain~$k = \frac{1-d}{d} n$ for the critical excess, 
which corresponds to a number of edges~$m = \frac{n}{d (d-1)}$,
a result already derived in~\cite{SS85}.

    \section{Kernels and Connected Hypergraphs} \label{sec:kernels}

In the seminal articles~\cite{W77} and~\cite{W80}, Wright establishes the connection
between the asymptotics of connected graphs with~$n$ vertices and excess~$k$ 
and the enumeration of the connected \emph{kernels}, 
which are multigraphs with no vertex of degree less than~$3$.
This relation was then extensively studied in~\cite{JKLP93}
and the notions of excess and kernels were extended to hypergraphs in~\cite{KL97}.

A kernel is a hypergraph with additional constraints that ensure that:
\begin{itemize}
  \item
    each hypergraph can be reduced to a unique kernel,
  \item
    the excess of a hypergraph that contain no tree component
    is equal to the excess of its kernel,
  \item
    for any integer~$k$, there is a finite number of kernels of excess~$k$,
  \item
    the generating function of hypergraphs of excess~$k$
    can be derived from the generating function of kernels of excess~$k$.
\end{itemize}
Observe that the two last requirements oppose each other:
the third one impose the kernels to be \emph{elementary},
but the fourth one means they should keep trace
of the structure of the hypergraph.
Following~\cite{KL97}, we define the \emph{kernel} of a hypergraph~$G$
as the result of the repeated execution of the following operations:
\begin{enumerate}
  \item 
    delete all the vertices of degree~$\leq 1$,
  \item 
    delete all the edges of size~$\leq 1$,
  \item 
    if two edges~$(a,v)$ and~$(v,b)$ of size~$2$
    have one common vertex~$v$ of degree~$2$,
    delete~$v$ and replace those edges with~$(a,b)$,
  \item 
    delete the connected components that consist
    of one vertex~$v$ of degree~$2$ and one edge~$(v,v)$ of size~$2$.
\end{enumerate}

The following lemma
has already been derived for uniform hypergraphs in~\cite{KL97}.
We give a new and more general proof. 
We also introduce the definition of \emph{clean} kernels,
borrowed from~\cite{KL97},
and derive an expression for their generating function.
As we will see in the proof of Theorem~\ref{th:complexe},
with high probability the kernel of a random hypergraph
in the critical window is clean.

\begin{lemma} \label{th:kernel}
  The number of kernels of excess~$k$ is finite
  and each of them contains at most~$3k$ edges of size~$2$.
  We say that a kernel is \emph{clean} if this bound is reached.
  The generating function of clean kernels with excess~$k$ is %
  \begin{equation} \label{eq:CleanKernel}
    e_k (1 + \omega_3 z^2)^{2k} \omega_2^{3k} z^{2k},
  \end{equation}
  where $e_k = \frac{(6k)!}{(3!)^{2k} 2^{3k} (3k)! (2k)!}$
  and the variables~$w$ and~$x$ have been omitted.
  The generating function of connected clean kernels with excess~$k$ is
  \begin{equation} \label{eq:connectedCleanKernel}
    c_k (1 + \omega_3 z^2)^{2k} \omega_2^{3k} z^{2k},
  \end{equation}
  where~$c_k = [z^{2k}] \log \sum_n e_n z^{2n}$.
\end{lemma}

\begin{proof}
  By definition of the excess, we have
  \[
    k+n+m = \sum_{e \in E} |e| = \sum_{v \in V} \deg(v).
  \]
  By construction, the vertices (resp. edges) of a kernel
  have degree (resp. size) at least~$2$, so
  \begin{align}
    k+n+m &\geq 3m - m_2, \label{eq:vdeg} \\
    k+n+m &\geq 3n - n_2, \label{eq:esize}
  \end{align}
  where~$n_2$ (resp.~$m_2$) is the number of vertices of degree~$2$ (resp. edges of size~$2$).
  Furthermore, each vertex of degree~$2$ belongs to an edge of size at least~$3$, so
  \begin{equation}
    k+n+m \geq 2m_2 + n_2. \label{eq:ev}
  \end{equation}
  Summing those three inequalities, we obtain~$3k \geq m_2$.

  This bound is reached if and only if~\eqref{eq:vdeg}, \eqref{eq:esize} and~\eqref{eq:ev}
  are in fact equalities. Therefore, the vertices (resp. edges) 
  of a clean kernel have degree (resp. size)~$2$ or~$3$, 
  each vertex of degree~$2$ belongs to exactly one edge of size~$3$
  and all the vertices of degree~$3$ belongs to edges of size~$2$.
  Consequently, any clean kernel can be obtained 
  from a cubic multigraph with~$2k$ vertices
  through substitutions of vertices of degree~$3$
  by groups of three vertices of degree~$2$ 
  that belong to a common edge of size~$3$.
  This construction is illustrated in Figure~\ref{fig:cubictocleankernel}.
  \begin{figure}
      \begin{center}
      \includegraphics[scale = 0.50]{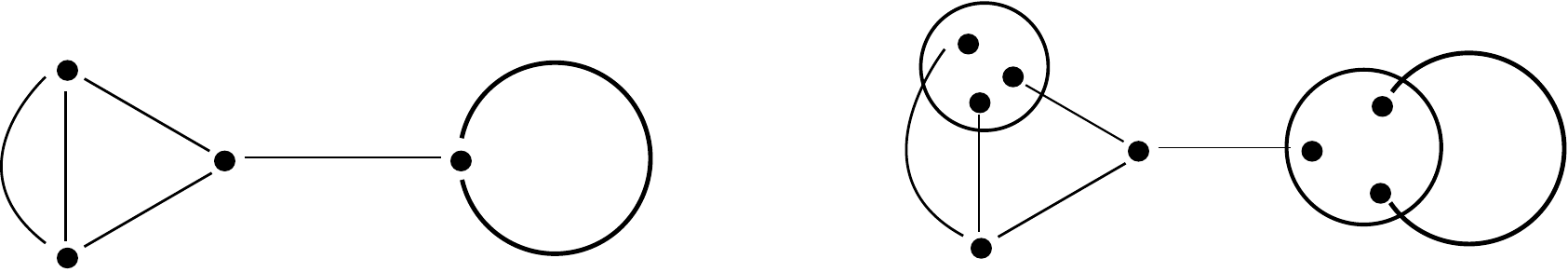}
      \end{center}
      \caption{A cubic multigraph and one of the clean kernels 
            that can be obtained from it. The labels have been omitted.}
      \label{fig:cubictocleankernel}
  \end{figure}
  Consequently, if~$f(z)$ denotes the generating function 
  of a family of cubic multigraphs,
  where~$z$ marks the vertices, then the generating function
  of the corresponding clean kernels is~$f(z + \omega_3 z^3)$.
  The number of cubic multigraphs of excess~$k$ 
  (\textit{i.e.} the sum of their compensation factors)
  is $(2k)! e_k$, so the generating function of cubic multigraphs of excess~$k$
  is~$\sum_{k \geq 0} e_k z^{2k}$.
  A cubic multigraph is a set of connected cubic multigraphs,
  so the value~$(2k)! c_k$ defined in the theorem 
  is the number of connected cubic multigraphs.

  To prove that the total number of kernels of excess~$k$ is bounded,
  we introduce the \emph{dualized} kernels, which are kernels where
  each edge of size~$2$ contains a vertex of degree at least~$3$.
  This implies the dual inequality of~\eqref{eq:ev}
  \[
    k+n+m \geq 2n_2 + m_2,
  \]
  which leads to~$7k \geq n + m$.
  Finally, each dualized kernel matches a finite number of normal kernels
  by replacing some vertices of degree~$2$ that do not belong to any edge of size~$2$
  with edges of size~$2$.
  This substitution is illustrated in Figure~\ref{fig:dualtonormal}.
    \begin{figure}
        \begin{center}
        \includegraphics[scale = 0.50]{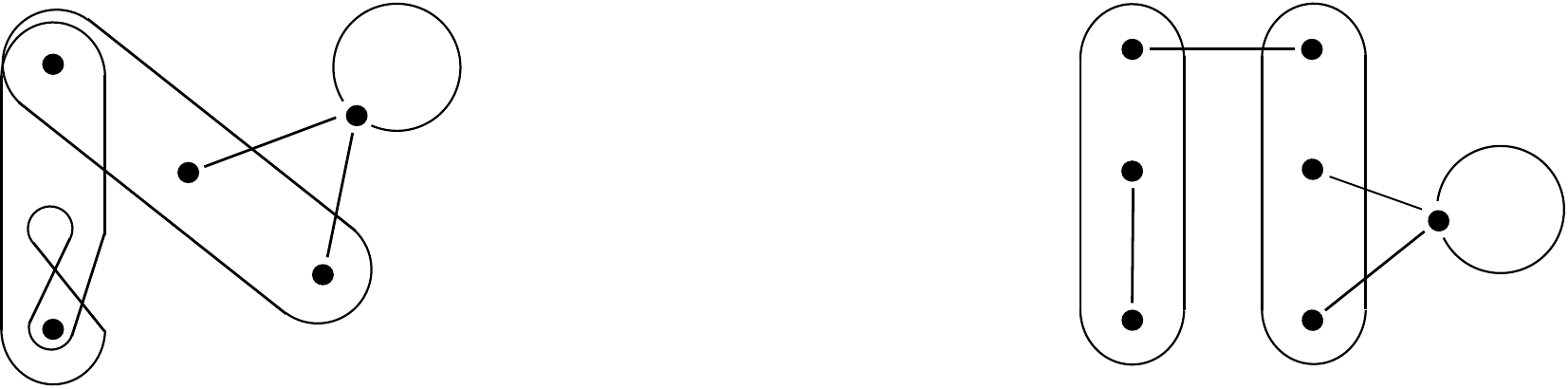}
        \end{center}
        \caption{A dualized kernel and one of the corresponding kernels. 
            The labels have been omitted.}
        \label{fig:dualtonormal}
    \end{figure}
\end{proof}

The previous lemma gives a way to construct 
all hypergraphs with fixed excess~$k$
that contain neither trees nor unicyclic components,
and all connected hypergraphs with fixed excess~$k$.
It starts from the finite corresponding set of kernels of excess~$k$,
adds rooted trees to its vertices, 
replaces the edges of size~$2$ with paths of trees
and adds rooted trees into the edges of size greater than~$2$.

\begin{lemma} \label{th:gfkernel}
    Let $\hpUV_k(z)$ denote the generating function
    of hypergraphs with excess~$k$ that contain
    neither trees nor unicyclic components,
    and $\chp_k(z)$ the generating function
    of connected hypergraphs with excess~$k$.
    With the notations of Lemma~\ref{th:kernel},
    we have
    \begin{align*}
      \hpUV_k(z) 
    &= \sum_{\ell=0}^{3k} 
      \frac{E_{k,\ell}(T(z))}{\left(1-T(z)\Omega''(T(z))\right)^\ell},
    \\
      \chp_k(z)
    &=
      \sum_{\ell=0}^{3k} 
        \frac{C_{k,\ell}(T(z))}{\left(1-T(z)\Omega''(T(z))\right)^\ell},
    \end{align*}
    where the functions $E_{k,\ell}(z)$ and $C_{k,\ell}(z)$ are entire functions and
    \begin{align*}
      E_{k,3k}(t) &= 
      e_k 
      \left( 1 + t^2 \Omega'''(t) \right)^{2k}
      t^{2k}
      \Omega''(t)^{3k},
        \\
      C_{k,3k}(t) &=
      c_k 
      \left( 1 + t^2 \Omega'''(t) \right)^{2k}
      t^{2k}
      \Omega''(t)^{3k}.
    \end{align*}
\end{lemma}

\begin{proof}
  Theorem~\ref{th:kernel} implies that the generating function
  of the kernels of excess~$k$ is a multivariate
  polynomial with variables~$z,\omega_2,\omega_3,\ldots$.
  Let us write it as the sum of two polynomials,
  $P_k$ and~$Q_k$, 
  one corresponding to clean kernels and the other
  to the rest of the kernels of excess~$k$.
  According to Theorem~\ref{th:kernel}, 
  $P_k$ is equal to~$e_k (1 + \omega_3 z^2)^{2k} \omega_2^{3k} z^{2k}$.
  By definition of the clean kernels,
  the degree of~$Q_k$ with respect to~$\omega_2$
  is strictly less than~$3k$.

  One can develop a kernel into a hypergraph
  by adding rooted trees to its vertices,
  replacing its edges of size~$2$ by paths of trees
  and adding rooted trees into the edges of size greater than~$2$.
  This matches the following substitutions in the generating function of kernels:
  $z \leftarrow T(z)$, $w_2 \leftarrow \frac{\weight''(T)}{1 - T \weight''(T)}$
  and~$w_t \leftarrow \weight^{(t)}(T)$ for all~$t > 2$.
  With this construction, starting with all kernels of excess~$k$, 
  we obtain all hypergraphs with excess~$k$ that contain 
  neither trees nor unicyclic components.
  Applying this substitution to~$P_k + Q_k$, 
  we obtain for their generating functions~$\hpUV_k(z)$
  \[
    \hpUV_k(z) 
    = 
    e_k 
    (1 + T(z)^2 \weight'''(T(z)))^{2k} 
    \left( \frac{\weight''(T(z))}{1 - T(z) \weight''(T(z))}\right)^{3k} 
    T(z)^{2k} 
    + \cdots
  \]
  where the ``$\cdots$'' hides terms
  with a denominator~$1 - T(z) \weight''(T(z))$
  at a power at most~$3k-1$.
  
  A hypergraph is connected if and only if its kernel is connected.
  The previous construction applied to connected kernels of excess~$k$
  leads to the similar following expression for the generating function
  of connected hypergraphs of excess~$k$
  \[
    \chp_k(z) 
    = 
    c_k 
    (1 + T(z)^2 \weight'''(T(z)))^{2k} 
    \left( \frac{\weight''(T(z))}{1 - T(z) \weight''(T(z))}\right)^{3k} 
    T(z)^{2k} 
    + \cdots
  \]
\end{proof}

From the previous lemma, it is easy to derive
the asymptotic number of connected hypergraphs with fixed excess.
The corresponding result for uniform hypergraphs
can be found in~\cite{KL97}.
As a corollary, those hypergraphs have clean kernels with high probability.

\begin{theorem} \label{th:connected}
  The number of connected hypergraphs 
  with~$n$ vertices and fixed excess~$k$ is
  \[
    n! [z^n] \chp_k(z) 
    \sim
    \frac{c_k  \sqrt{2 \pi}}{\Gamma\left( \frac{3k}{2} \right)}
    \left( \frac{\sqrt{\gamma}}{2^{3/2} \tau} \right)^k
    \left( \frac{e}{\rho} \right)^n
    n^{n + (3k-1)/2},
  \]
  where~$c_k$ is defined in Lemma~\ref{th:kernel},
  $\tau$ is the solution of~$\tau \weight''(\tau) = 1$,
  $\rho = \tau e^{- \weight'(\tau)}$
  and~$\gamma = 1 + \tau^2 \weight'''(\tau)$.
  The same results apply to connected simple hypergraphs.
\end{theorem}

\begin{proof}
  Injecting the Puiseux expansion of the generating function of rooted trees
  \[
    T(z) \sim \tau - \tau \sqrt{\frac{2}{\gamma}} \sqrt{1-z/\rho}
  \]
  in the expression of $\chp_k(z)$ derived in Lemma~\ref{th:gfkernel}, we obtain
  \[
    \chp_k(z) \sim c_k \left( \frac{\sqrt{\gamma}}{2^{3/2} \tau} \right)^k (1-z/\rho)^{-3k/2}.
  \]
  The asymptotic enumeration result follows by singularity analysis
  \cite[Theorem~VI.4]{FS09}.

  We now prove that the result holds for connected simple hypergraphs.
  As shown in the first part of the proof,
  we can restrict our investigation to hypergraphs with clean kernels.
  Among them, let us consider the set of connected hypergraphs with excess~$k$
  that are not simple. 
  Each one contains a loop or two edges of size~$2$ linking the same vertices.
  Therefore, the kernel of each of them has at least
  one edge of size~$2$ 
  that is not replaced by a (non-empty) path of threes in the hypergraph.
  It follows that the generating function of those hypergraphs
  has a denominator~$1 - T(z) \weight''(T(z))$ at a power at most~$3k-1$,
  so the cardinality of this set is negligible
  compared to the number of connected hypergraphs with excess~$k$.

  Another and more intuitive way to understand it
  is that at fixed excess, adding more and more vertices
  into a kernel, the chances that an edge of size~$2$ 
  does not break into a non-empty path of threes are negligible.
\end{proof}

To derive a complete asymptotic expansion of connected hypergraphs,
one needs to take into account non-clean kernels.
For any fixed~$k$, 
one can enumerate all the kernels of excess~$k$
(since it is a finite set),
then apply the substitution described in the previous proof
to obtain the generating function of all connected hypergraphs
of excess~$k$, from which a complete asymptotic expansion follows.
Although computable, this construction is heavy.
The purely analytic approach of~\cite{FSS04},
that addresses this problem for graphs,
may allow a simpler expression.

The asymptotic enumeration of connected hypergraphs in~$\HP_{n,k}$
when~$k$ tends toward infinity is more challenging.
Since the original result for graphs from Bender, Canfield and McKay~\cite{BCM90},
other proofs have been proposed by Pittel and Wormald~\cite{PW05} 
and van der Hofstad and Spencer~\cite{HS06},
which may be generalized to hypergraphs.

    \section{Structure of Hypergraphs in the Critical Window} \label{sec:hpcomplex}

The next theorem describes the structure of hypergraphs
with an excess at or close to the critical value~$k = (\crit - 1) n$
introduced in Theorem~\ref{th:shp}.
It generalizes Theorem~$5$ of~\cite{JKLP93} about graphs.
Interestingly, the result at the critical excess
does not depend on the~$(\omega_t)$.

\begin{theorem} \label{th:complexe}
  Let~$\Psi(z)$, $\tau$, $\crit$ and~$\gamma$
  be defined as in Theorem~\ref{th:shp}.
  Let~$r_1, \ldots, r_q$ denote a finite sequence of integers
  and~$r = \sum_{t=1}^q t\, r_t$, then
  the limit of the probability for a hypergraph or simple hypergraph 
  with~$n$ vertices and global excess~$k = (\crit - 1) n + \bigO(n^{1/3})$ 
  to have exactly~$r_t$ components of excess~$t$ for~$t$ from~$1$ to~$q$ is
  \begin{equation} \label{eq:probar}
    \left(\frac{4}{3} \right)^r  
    \frac{r!}{(2r)!}
    \sqrt{\frac{2}{3}}
    \frac{c_1^{r_1}}{r_1!} \frac{c_2^{r_2}}{r_2!} \cdots \frac{c_q^{r_q}}{r_q!}.
  \end{equation}
  where the~$(c_i)$ are defined as in Theorem~\ref{th:kernel}.
  For~$k = (\crit - 1) n + x n^{2/3}$ and~$x$ bounded, the limit of this probability is
  \[
    3^{-r} 
    \frac{c_1^{r_1}}{r_1!} \frac{c_2^{r_2}}{r_2!} \cdots \frac{c_q^{r_q}}{r_q!}
    \sqrt{\frac{3 \pi}{2}} 
    \exp\left({\frac{-x^3 \gamma}{6(1-\crit)^3}} \right)
    G \left( \frac{3}{2}, \frac{1}{4} + \frac{3 r}{2} ; - \frac{3^{2/3} \gamma^{1/3} x}{2 (1 - \crit)} \right).
  \]
\end{theorem}

\begin{proof}
  Let~$\chp_k(z)$ denote the generating function 
  of connected hypergraphs of excess~$k$. 
  From Theorem~\ref{th:connected}, 
  when~$z$ tends towards the dominant singularity~$\rho$ of~$T(z)$,
  \[
    \chp_k(z) 
  \sim 
    c_k\left( \frac{\sqrt{\gamma}}{2^{3/2} \tau} \right)^k (1 - z/\rho)^{-3k/2}.
  \]
  The sum of the weights of hypergraphs with global excess~$k$
  and~$r_t$ components of excess~$t$ is
  \[
    n! [z^n] \frac{U^{r - k}}{(r-k)!} e^V 
    \frac{\chp_1(z)^{r_1}}{r_1!} \frac{\chp_2(z)^{r_2}}{r_2!} \ldots \frac{\chp_q(z)^{r_q}}{r_q!}
  \]
  and an application of Theorem~\ref{th:coal} ends the proof,
  with~$G(3/2,1/4 + 3 r/2; 0) = \frac{2}{3\sqrt{\pi}} \frac{4^r r!}{(2r)!}$.
  Those computations are the same as in Theorem~\ref{th:shp}.
\end{proof}

We have derived the limit of the probability for the \emph{structure} of a random hypergraph,
\textit{i.e.} the number of components of each finite excess.
However, would the hypergraph contain a component 
with an excess going to infinity with the number $n$ of vertices,
those limit of probabilities could not capture it.
Therefore, we now need to prove that this situation has a zero limit probability.
To do so, we prove that the limit of probabilities we derived
form in fact a probability distribution, meaning that they sum to $1$.
In~\cite[p.~$52$]{JKLP93}, the authors prove 
that the sum over all $r_1, r_2, \ldots \geq 0$
of Expression~\eqref{eq:probar} is equal to~$1$.
This corresponds to the special case~$x=0$ of the next theorem.

\begin{theorem}
    We consider a random hypergraph in $\HP_{n,k}$ or $\NMHP_{n,k}$
    with $k = (\Lambda - 1) n + x n^{2/3}$ and $x$ bounded.
    With high probability, all its components have a finite excess.
\end{theorem}

\begin{proof}
According to Theorem~\ref{th:complexe}, 
the theorem is proved once the following equality is established
\[
	\sum_{q \geq 0}
	\sum_{r_1, \ldots, r_q \geq 0}
	3^{-r}
	\frac{c_1^{r_1}}{r_1!} \cdots \frac{c_q^{r_q}}{r_q!}
	\sqrt{\frac{3 \pi}{2}}
	e^{\frac{-x^3 \gamma}{6(1-\Lambda)^3}}
	G\left( \frac{3}{2}, \frac{1}{4} + \frac{3 r}{2};
		- \frac{3^{2/3} \gamma^{1/3} x}{2 (1-\Lambda)}
	\right) = 1,
\]
where $r = \sum_{t=1}^q t\, r_t$.
We rescale the variable~$x$ and replace the function~$G$ with
\[
	G\left( \frac{3}{2}, \frac{1}{4} + \frac{3 r}{2};
		- \frac{3^{2/3} \gamma^{1/3} x}{2 (1-\Lambda)}
	\right)
	=
	\frac{2}{3}
	\sum_{k \geq 0}
	\frac{1}{\Gamma\left( \frac{1}{2} + r - \frac{2 k}{3}\right)}
	\frac{x^k}{k!},
\]
an expression derived using, on the definition of~$G$ from Theorem~\ref{th:coal}, 
the complement formula for the Gamma function
\[
    \Gamma(s) \Gamma(1-s) = \frac{\pi}{\sin(\pi s)}.
\]
The equality we want to prove becomes
\begin{align} \label{eq:telescoping}
	&\sum_{q \geq 0}
	\sum_{r_1, \ldots, r_q \geq 0}
	3^{-r}
	\frac{c_1^{r_1}}{r_1!} \cdots \frac{c_q^{r_q}}{r_q!}
	F_r(x)
	=
	\sqrt{\frac{3}{2 \pi}}
	e^{4 x^3 / 27},
	\\
	\text{where }\quad &
	F_r(x) =
	\sum_{k \geq 0}
	\frac{1}{\Gamma\left( \frac{1}{2} + r - \frac{2 k}{3}\right)}
	\frac{x^k}{k!}. \nonumber
\end{align}
To simplify this expression further,
we observe that the coefficient extraction 
in the exponential of a generating function 
has a familiar shape
\[
    [z^r] e^{\sum_{i \geq 1} c_i z^i} =
    \sum_{q \geq 1\ }
    \sum_{r_1 + 2 r_2 + \cdots + q r_q = r\ }
    \prod_{i=1}^q \frac{c_i^{r_i}}{r_i !}.
\]
The generating function of the $(c_i)$ is characterized by
\[
    \sum_{i \geq 1} c_i z^{2i} = 
    \log \Biggl( \sum_{r \geq 0} e_r z^{2 r} \Biggr),
\]
where the value of $e_r$ is
\[
    e_r = \frac{(6r)!}{(3!)^{2r} 2^{3r} (3r)! (2r)!}.
\]
Therefore,
\[
	\sum_{q \geq 0}
	\sum_{r_1, \ldots, r_q \geq 0}
	3^{-r}
	\frac{c_1^{r_1}}{r_1!} \cdots \frac{c_q^{r_q}}{r_q!}
	F_r(x)
	=
	\sum_{r \geq 0}
	3^{-r}
	e_r
	F_r(x)
\]
and Equation~\eqref{eq:telescoping}, which we want to prove, is equivalent with
\begin{equation} \label{eq:telescopingdeux}
	\sum_{r \geq 0}
	3^{-r}
	e_r
	F_r(x)
	=
	\sqrt{\frac{3}{2 \pi}}
	e^{4 x^3 / 27}.
\end{equation}

We first prove the theorem in the particular case $x=0$,
which corresponds to $k = (\Lambda - 1) n$.
Since we have
\[
	F_r(0) 
	= 
	\frac{1}{\Gamma(r + 1/2)} 
	= 
	\frac{4^r}{\sqrt{\pi}}
	\frac{r!}{(2r)!},
\]
Equation~\eqref{eq:telescopingdeux} becomes
\[
	\sum_{r \geq 0}
	3^{-r}
	\frac{(6r)!}{(3!)^{2r} 2^{3r} (3r)! (2r)!}
	4^r
	\frac{r!}{(2r)!}
	=
	\sqrt{\frac{3}{2}}.
\]
Let~$u_r$ denote the summand of the left side, then
\[
    \frac{u_{r+1}}{u_r} =
    \frac{(r + \frac{1}{6}) (r + \frac{5}{6})}{(r+ \frac{1}{2}) (r+1)} 
    \frac{1}{2}.
\]
It follows that the sum is equal to the evaluation of a hypergeometric function 
at the point~$1/2$
\[
    \sum_{r \geq 0}
	3^{-r}
	\frac{(6r)!}{(3!)^{2r} 2^{3r} (3r)! (2r)!}
	4^r
	\frac{r!}{(2r)!} 
	= 
	{}_2\!F_1\left(\frac{1}{6}, \frac{5}{6}, \frac{1}{2}; \frac{1}{2} \right).
\]
A special identity of the hypergeometric function states
\[
    {}_2\!F_1(a, 1-a, c; 1/2) =
    \frac{\Gamma(c/2) \Gamma((c+1)/2)}{\Gamma((a+c)/2) \Gamma((1+c-a)/2)},
\]
which reduces, for~$a=1/6$ and~$c=1/2$, to
\[
	{}_2\!F_1\left(\frac{1}{6}, \frac{5}{6}, \frac{1}{2}; \frac{1}{2} \right)
	= \sqrt{\frac{3}{2}}
\]
and achieves the proof of the theorem in the particular case $x=0$.

We use computer algebra to solve the general case,
specifically Koutschan's\footnote{We thank Christoph Koutschan for his help.} Mathematica package~\cite{Kpackage}.
Let~$\Delta_r$ denote the operator on sequences
\[
	\Delta_r(u_r) = u_{r+1} - u_r
\]
and~$\partial$ the differential operator with respect to~$x$.
We first apply the creative telescoping algorithm (see~\cite{K10} for more details) 
to the sum
\[
	F_r(x) =
	\sum_{k \geq 0}
	\frac{1}{\Gamma\left( \frac{1}{2} + r - \frac{2 k}{3}\right)}
	\frac{x^k}{k!}
\]
It computes pairs of operators of the form
\[
	\left( P(x, r, \partial, \Delta_r), Q(x, r, k, \partial, \Delta_r, \Delta_k) \right)
\]
such that
\[
	P \cdot \left( 
		\frac{1}{\Gamma\left( \frac{1}{2} + r - \frac{2 k}{3}\right)}
		\frac{x^k}{k!}
	\right)
	+
	\Delta_k Q \cdot
	\left( 
		\frac{1}{\Gamma\left( \frac{1}{2} + r - \frac{2 k}{3}\right)}
		\frac{x^k}{k!}
	\right)
	= 0.
\]
Since~$P$ is independent of~$k$, it commutes with a summation over~$k$
and we obtain
\[
	P \cdot F_r(x) + 
	\left[ 
		Q \cdot
		\Bigg( 
			\frac{1}{\Gamma\left( \frac{1}{2} + r - \frac{2 k}{3}\right)}
			\frac{x^k}{k!}
		\Bigg)
	\right]_{k=0}^{+ \infty} = 0
\]
We then check that the right side of the sum cancels for each operator~$Q$
returned by the algorithm.
The family $(P)$ is a Gr\"obner basis for the ideal of Ore operators 
in~$\Delta_r$ and~$\partial$ that cancels~$F_r(x)$.
From there, using again Koutschan's package, 
we compute a Gr\"obner basis for the ideal of Ore operators that cancels
\[
	3^{-r} e_r F_r(x).
\]
We then apply the creative telescoping algorithm with this basis as input.
The output is a pair of operators
\[
	\left(
		9 \partial - 4 x^2,
		Q(x,r, \partial, \Delta_r)
	\right)
\]
such that
\[
	\left( 9 \partial - 4 x^2\right) \cdot 
		\Bigg( \sum_{r \geq 0} 3^{-r} e_r F_r(x)) \Bigg)
	+ \Bigg[
		Q \cdot \left( 3^{-r} e_r F_r(x) \right)
	\Bigg]_{r=0}^{+\infty}
	= 0.
\]
After verification that the right side of the sum cancels, 
we conclude that the operator~$9 \partial - 4 x^2$ 
cancels~$\sum_{r \geq 0} 3^{-r} e_r F_r(x))$.
All solutions of this differential equation have the form
\[
	C e^{4 x^3/27}
\]
and the initial condition corresponding to~$x=0$
fixes the constant~$C = \sqrt{3/(2 \pi)}$.
This proves Equation~\eqref{eq:telescopingdeux}.
\end{proof}

As a corollary, the previous theorem implies
that Theorem~\ref{th:complexe} holds true for unbounded~$q$.

    \section{Birth of the Giant Component} \label{sec:giant}

Erd\H{o}s and R\'enyi~\cite{ER60} analyzed
random graphs with a large number $n$ of vertices
and $m$ edges such that $m/n$ tends toward a constant $c$.
They proved that when $c$ is strictly greater than $1/2$,
with high probability the graph has a \emph{giant component},
which contains a constant fraction of the vertices
and has an excess going to infinity with $n$.
Similar results have been derived for various models of hypergraphs
\cite{SS85}, \cite{KL02}, \cite{Co04}, \cite{DM08}, \cite{GZVCN09}.
%

%

We consider random hypergraphs with $n$ vertices
and excess $k = (\lambda - 1) n$.
We have seen in Theorem~\ref{th:forest}
that when $\lambda < \Lambda$,
with high probability the hypergraph
contains only trees and unicyclic components.
We also have derived the limit distribution of the excesses
of the components in the critical window, \textit{i.e.}
for hypergraphs with excess $k = (\Lambda - 1) n + x n^{2/3}$ with $x$ bounded.
In this section, we investigate the case $\lambda > \Lambda$.
However, we will not derive results on the giant component as precise
as Erd\H{o}s and R\'enyi~\cite{ER60} did for graphs, using different tools.
We first prove that, with high probability,
the hypergraph contains a component with unbounded excess.
We conjecture that this component is unique and contains 
a constant fraction of the vertices.
Therefore, we refer to it as \emph{the giant component}.

\begin{theorem}
    With the notations of Theorem~\ref{th:forest},
    let us consider a random hypergraph
    with $n$ vertices and excess $k = (\lambda - 1) n$,
    where $\lambda$ is strictly greater than $\Lambda$.
    For any fixed $K$, the limit probability that all its components
    have excess smaller or equal to $K$ is zero.
\end{theorem}

\begin{proof}
Let $\hp_{n,k}^{(\ell)}$ denote the number of hypergraphs 
with $n$ vertices, excess~$k$ and kernel of excess~$\ell$.
Since each tree has excess~$(-1)$, such a hypergraph contains
a set of $\ell-k$ trees, so
\[
  \hp_{n,k}^{(\ell)} = n! [z^n] 
  \frac{U(z)^{\ell-k}}{(\ell-k)!}
  e^{V(z)}
  \hpUV_\ell(z).
\]
We replace $V(z)$ with its expression
$\log\left( \frac{1}{1- T(z) \Omega''(T(z))} \right)$
and $\hpUV_{\ell}(z)$ with the expression derived in Lemma~\ref{th:gfkernel}
\[
  \hp_{n,k}^{(\ell)} = 
  \sum_{j = 0}^{3 \ell}
  \frac{n!}{(\ell-k)!} [z^n]
  \frac{U(z)^{\ell-k}}{
    (1-T(z) \Omega''(T(z)))^{j+\frac{1}{2}}}
  E_{\ell,j}(T(z)).
\]
We introduce the notation $\hp_{n,k}^{(\ell,j)}$ such that
$  \hp_{n,k}^{(\ell)} = 
  \sum_{j = 0}^{3 \ell}
  \hp_{n,k}^{(\ell,j)}$
\[
  \hp_{n,k}^{(\ell,j)} =
  \frac{n!}{(\ell-k)!} [z^n]
    \frac{U(z)^{-k}}{
      (1-T(z) \Omega''(T(z)))^{j+\frac{1}{2}}}
    U(z)^{\ell} E_{\ell,j}(T(z)).
\]
By definition of $\Psi(t) = \Omega'(t) - \frac{\Omega(t)}{t}$
and $U(z) = T(z) + \Omega(T(z)) - T(z) \Omega'(T(z))$,
we have $U(z) = (1 - \Psi(T(z))) T(z)$.
Therefore, using the Cauchy integral representation of the coefficient extraction,
we obtain that~$\hp_{n,k}^{(\ell,j)}$ is equal to
\[
  \frac{n!}{(\ell-k)!}
  \frac{1}{2 i \pi}
  \oint
  \frac{(1-\Psi(T(z)))^{-k} T(z)^{-k} }{(1-T(z) \Omega''(T(z)))^{j+\frac{1}{2}}}
  (1-\Psi(T(z)))^\ell T(z)^\ell E_{\ell,j}(T(z))
  \frac{dz}{z^{n+1}}.
\]
After the change of variable $t = T(z)$, this expression becomes
\[
  \hp_{n,k}^{(\ell,j)} =
  \frac{n!}{(\ell-k)!}
  \frac{1}{2 i \pi}
  \oint
  \frac{(1-\Psi(t))^{-k} e^{\Omega'(t) n}}{(1-t\Omega''(t))^{j - \frac{1}{2}}}
  (1-\Psi(t))^\ell t^\ell E_{\ell,j}(t)
  \frac{dt}{t^{n+k+1}}.
\]
Let us recall that $k = (\lambda - 1) n$, so we introduce the notation
\[
  \Phi_{\lambda}(t) =
  (1-\Psi(t))^{\frac{1-\lambda}{\lambda}}
  e^{\frac{\Omega'(t)}{\lambda}}
\]
in order to rewrite the expression of $\hp_{n,k}^{(\ell,j)}$
\[
  \hp_{n,k}^{(\ell,j)} =
  \frac{n!}{(\ell-k)!}
  \frac{1}{2 i \pi}
  \oint
  \frac{ \Phi_{\lambda}(t)^{n+k} }{ (1-t\Omega''(t))^{j - \frac{1}{2}} }
  (1-\Psi(t))^\ell t^\ell E_{\ell,j}(t)
  \frac{dt}{t^{n+k+1}}.
\]
At this point, it is tempting to use a saddle-point method
on a circle to extract the asymptotics.
The dominant saddle-point $\zeta$ would be characterized 
as the smallest root of
\[
  1 - \zeta \frac{\Phi_{\lambda}'(\zeta)}{\Phi_{\lambda}(\zeta)},
\]
which is equal with
\[
  (1-t\Omega''(t)) 
  \frac{1}{\lambda}
  \frac{\lambda - \Psi(t)}{1 - \Psi(t)}.
\]
Since $\lambda$ is greater than $\Lambda = \Psi(\tau)$, 
this saddle-point is equal to $\tau$,
and thus cancels the denominator $(1 - t \Omega''(t))$.
This is a case of coalescence of a saddle-point with a singularity.
To overcome this situation, 
following the example given in \cite[Note VIII.39]{FS09},
we apply the change of variable $t = s \Phi_\lambda(t)$
\[
  \hp_{n,k}^{(\ell,j)} =
  \frac{n!}{(\ell-k)!}
  \frac{1}{2 i \pi}
  \oint
  \frac{ 1 }{ (1-t\Omega''(t))^{j + \frac{1}{2}} }
  \frac{\lambda (1-\Psi(t))}{\lambda - \Psi(t)}
  (1-\Psi(t))^\ell t^\ell E_{\ell,j}(t)
  \frac{ds}{s^{n+k+1}},
\]
where $t$ is now considered as a function of $s$.
This integral can be interpreted as a coefficient extraction
\begin{equation} \label{eq:hpnklj}
  \hp_{n,k}^{(\ell,j)} =
  \frac{n!}{(\ell-k)!}
  [s^{n+k}]
  \frac{ 1 }{ (1-t\Omega''(t))^{j + \frac{1}{2}} }
  \frac{\lambda (1-\Psi(t))^{\ell+1} t^\ell }{\lambda - \Psi(t)}
  E_{\ell,j}(t).
\end{equation}
Applying~\cite[Theorem~VI.6]{FS09} to the expression~$t = s \Phi_{\lambda}(t)$,
we compute the first terms of the singular expansion of~$t(s)$
\[
	t(s) \sim
	\tau - d_1 \sqrt{1-s/\rho_2}
\]
where the values~$\rho_2$ and~$d_1$ are given by
\[
	\rho_2 = \frac{\tau}{(1-\Lambda)^{\frac{1-\lambda}{\lambda}} e^{\frac{\Omega(\tau)}{\lambda \tau} + \frac{\Lambda}{\lambda}}}, 
	\quad \text{ and} \quad
	d_1 = \sqrt{\frac{2 \Phi_{\lambda}(\tau)}{\Phi_{\lambda}''(\tau)}}.
\]
It follows that
\begin{align*}
	\frac{ 1 }{ (1-t\Omega''(t))^{j + \frac{1}{2}} }
  	&\frac{\lambda (1-\Psi(t))^{\ell+1} t^\ell }{\lambda - \Psi(t)}
	E_{\ell,j}(t)
	\sim\\
	&(1-s/\rho_2)^{-\left(\frac{j}{2} + \frac{1}{4}\right)}
	\left( \frac{\tau}{\gamma d_1} \right)^{j + \frac{1}{2}}
	\frac{\lambda (1-\Lambda)^{\ell+1} \tau^\ell }{\lambda - \Lambda} 
	E_{\ell,j}(\tau).
\end{align*}
In Equation~\eqref{eq:hpnklj}, we use Stirling formula to express
\[
	\frac{n!}{(-k)!} \sim
	\frac{n^{n+k}}{\sqrt{1-\lambda}}
	\frac{e^{-\lambda n}}{(1-\lambda)^{(1-\lambda) n}},
\]
and derive the asymptotics from the coefficient extraction 
using a singularity analysis \cite[Theorem VI.4]{FS09}
\[
	\hp_{n,k}^{\ell,j} \sim
	n^{n+k}
	\left(
		\left( \frac{1-\Lambda}{1-\lambda}\right)^{1-\lambda}
		e^{\Lambda - \lambda}
		\frac{e^{\frac{\Omega(\tau)}{\tau}}}{\tau^{\lambda}}
	\right)^n
	n^{\frac{j}{2} - \frac{3}{4}}
	C_{\ell, j, \lambda}
\]
where the value~$C_{\ell, j, \lambda}$ is positive and bounded with respect to~$n$.
We conclude that the dominant term in the sum
\[
	\hp_{n,k}^{\ell} = \sum_{j=0}^{3 \ell} \hp_{n,k}^{\ell,j}
\]
is~$\hp_{n,k}^{\ell, 3 \ell}$
and that the number of hypergraphs with all components with excess at most~$K$
is equivalent with
\[
	\sum_{\ell = 0}^{K} \hp_{n,k}^{\ell} \sim
	\hp_{n,k}^{K} \sim
	n^{n+k}
	\left(
		\left( \frac{1-\Lambda}{1-\lambda}\right)^{1-\lambda}
		e^{\Lambda - \lambda}
		\frac{e^{\frac{\Omega(\tau)}{\tau}}}{\tau^{\lambda}}
	\right)^n
	n^{\frac{3}{2}K - \frac{3}{4}}
	C_{K, 3K, \lambda}.
\]
The asymptotic probability for a hypergraph
to have all components of excess at most~$K$ is then
the previous value divided by the total number of hypergraphs,
derived in Theorem~\ref{th:hp}
\[
	\left(
		\left( \frac{1-\Lambda}{1-\lambda}\right)^{1-\lambda}
		e^{\Lambda - \lambda}
		\left( \frac{\zeta}{\tau} \right)^{\lambda}
		e^{\frac{\Omega(\tau)}{\tau} - \frac{\Omega(\zeta)}{\zeta}}
	\right)^n
	n^{\frac{3}{2} K - \frac{1}{4}}
	\sqrt{2 \pi (\zeta \Omega''(\zeta) - \lambda)}
	C_{K,3K,\lambda}
\]
where~$\zeta$ is characterized by~$\Psi(\zeta) = \lambda$.
There are two ways to prove that this quantity goes to zero when~$n$ is large.
The first one is to establish the inequality
\[
	\left( \frac{1-\Lambda}{1-\lambda}\right)^{1-\lambda}
		e^{\Lambda - \lambda}
		\left( \frac{\zeta}{\tau} \right)^{\lambda}
		e^{\frac{\Omega(\tau)}{\tau} - \frac{\Omega(\zeta)}{\zeta}}
	< 1
\]
using the inequalities~$\Lambda < \lambda$ and~$\tau < \zeta$.
This can be achieved by successive derivations of the logarithm of the expression.
The second and simpler one uses the observation that this quantity can only tend
to~$0$ or to~$+ \infty$. Since it is a probability, the second option is impossible.
\end{proof}

Molloy and Reed~\cite{MR95} and Newman, Strogatz and Watts~\cite{NSW01}
gave an intuitive explanation of the birth of the giant component
in graphs with known degree distribution\footnote{We thank an anonymous referee 
for those references.}.
Starting with a vertex, we can determine the component in which it lies
by exploring its neighbors, then the neighbors of its neighbors
and so on. This branching process is likely to stop rapidly
if the expected number of new neighbors is smaller than~$1$.
On the other hand, the component is likely to be large
if this means is greater than~$1$.

We now explain why the expected number of new neighbors is smaller than $1$
only for subcritical hypergraphs.
Let us define the \emph{excess degree} of a vertex $v$ in an edge $e$
as the sum over all the other edges that contain $v$
of their sizes minus $1$
\[
    \operatorname{excess\ degree}(v,e) = 
    \sum_{\substack{v \in \tilde{e},\\ \tilde{e} \neq e}}
    \left( |\tilde{e}| - 1 \right).
\]
This is the number of neighbors we discover
when we arrive at the vertex $v$ from the edge $e$,
assuming they are distinct.
We now prove that the expected excess degree is smaller than~$1$
only for subcritical hypergraphs.
This provides an intuitive explanation for the birth of the giant component.

\begin{theorem}
    Let us consider a random hypergraph
    with $n$ vertices and excess $k = (\lambda - 1) n$,
    and a uniformly chosen pair $(v,e)$,
    where the vertex $v$ belongs to the edge~$e$.
    With the notations of Theorem~\ref{th:forest},
    the expected excess degree of $v$ in $e$
    is smaller than (resp. equal to or greater than) $1$,
    if $\lambda$ is smaller than (resp. equal to or greater than) $\Lambda$.
\end{theorem}

\begin{proof}
    Let $F(u)$ denote the generating function
    of the degree excess of a marked vertex in a marked edge
    of a hypergraph with $n$ vertices and excess $k = (\lambda - 1) n$.
    The marked vertex and edge represent $v$ and $e$.
    Such a hypergraph can be decomposed into
    a hypergraph on $n-1$ vertices,
    an edge with one vertex marked 
    -- this is the edge $e$ that contains the vertex $v$ --
    and a set of edges with one vertex marked and another vertex replaced with a hole.
    Those last edges contain the neighbors of $v$
    that are counted by the excess degree.
    Introducing the variable $u$ to mark the excess degree of $v$
    and the variable $y$ for the excess of the hypergraph, we obtain
    \[
        F(u) = [y^{n+k}]
        e^{\frac{\Omega((n-1) y)}{y}}
        \Omega'(ny)
        e^{\Omega'(n y u)}.
    \]
    After the change of variable $n y \to y$, this expression becomes
    \[
        F(u) = n^{n+k} [y^{n+k}]
        e^{n \frac{\Omega \left(y - \frac{y}{n} \right)}{y}}
        \Omega'(y)
        e^{\Omega'(y u)},
    \]
    which can be approximated by
    \[
        F(u) = n^{n+k} [y^{n+k}]
        e^{n \frac{\Omega(y)}{y} - \Omega'(y) + \bigO(1/n)}
        \Omega'(y)
        e^{\Omega'(y u)}.
    \]
    The expected excess degree of $v$ in $e$ is then $F'(1)/F(1)$
    (see~\cite[Part C]{FS09})
    \[
        \mathds{E}(\operatorname{excess\ degree}) =
        \frac{F'(1)}{F(1)} =
        \frac{[y^{n+k}] 
            e^{n \frac{\Omega(y)}{y} - \Omega'(y) + \bigO(1/n)}
            \Omega'(y)
            y \Omega''(y)
            e^{\Omega'(y)}}
        {[y^{n+k}] 
            e^{n \frac{\Omega(y)}{y} - \Omega'(y) + \bigO(1/n)}
            \Omega'(y)
            e^{\Omega'(y)}}.
    \]
    The asymptotics of the coefficient extractions in the computation
    of $F(1)$ and $F'(1)$ are obtained using the Large Powers Theorem
    \cite[Theorem~VIII.8]{FS09}.
    The saddle-point $\zeta$ is characterized by
    \[
        \Psi(\zeta) = \lambda,
    \]
    and the limit value of the expectation is
    \[
        \lim_{n \to \infty}
        \mathds{E}(\operatorname{excess\ degree})
        =
        \zeta \Omega''(\zeta).
    \]
    Let us recall the equalities $\Lambda = \Psi(\tau)$ 
    and $\tau \Omega''(\tau) = 1$.
    Since $\Psi(z)$ and $z \Omega''(z)$ are increasing functions,
    it follows that
    when $\lambda$ is smaller than (resp. equal to or greater than) $\Lambda$,
    then $\zeta \Omega''(\zeta)$ is smaller than (resp. equal to or greater than)~$1$.
\end{proof}

    \section{Future Directions}


In the present paper, for the sake of the simplicity of the proofs,
we restrained our work to the case
where~$e^{\Omega(z)/z}$ is aperiodic.
This technical condition can be waived
in the same way Theorem~VIII.8 of~\cite{FS09} can be extended
to periodic functions.

In the model we presented, the weight~$\omega_t$ of an edge
only depends on its size~$t$. For some applications,
one may need weights that also vary with the number of vertices~$n$.
It would be interesting to measure the impact of this modification
on the phase transition properties described in this paper.

More generally, the study of the relation to other models,
as the one presented in~\cite{DN04} and~\cite{BJR08}, 
could lead to new developments and applications.



\bibliography{/home/elie/research/articles/bibliography/biblio}

\begin{thebibliography}{10}

\bibitem{BFSS01}
Cyril Banderier, Philippe Flajolet, Gilles Schaeffer, and Mich{\`e}le Soria.
\newblock Random maps, coalescing saddles, singularity analysis, and {A}iry
  phenomena.
\newblock {\em Random Structures Algorithms}, 19(3-4):194--246, 2001.

\bibitem{BCM90}
Edward~A. Bender, E.~Rodnay Canfield, and Brendan~D. McKay.
\newblock The asymptotic number of labeled connected graphs with a given number
  of vertices and edges.
\newblock {\em Random Structures and Algorithm}, 1:129--169, 1990.

\bibitem{B85}
C.~Berge.
\newblock {\em Graphs and Hypergraphs}.
\newblock Elsevier Science Ltd, 1985.

\bibitem{BLL97}
Fran\c{c}ois Bergeron, Gilbert Labelle, and Pierre Leroux.
\newblock {\em Combinatorial Species and Tree-like Structures}.
\newblock Cambridge University Press, 1997.

\bibitem{BJR08}
B\'ela Bollob\'as, Svante Janson, and Oliver Riordan.
\newblock Sparse random graphs with clustering.
\newblock {\em Random Structures and Algorithms}, 38(3):269--323, 2011.

\bibitem{Co04}
Colin Cooper.
\newblock The cores of random hypergraphs with a given degree sequence.
\newblock {\em Random Struct. Algorithms}, 25(4):353--375, 2004.

\bibitem{DN04}
Richard W.~R. Darling and James~R. Norris.
\newblock Structure of large random hypergraphs.
\newblock {\em Annals of Applied Probability}, 23(6):125--152, 2004.

\bibitem{DM08}
Amir Dembo and Andrea Montanari.
\newblock Finite size scaling for the core of large random hypergraphs.
\newblock {\em Ann. Appl. Probab.}, 18(5):1993--2040, 10 2008.

\bibitem{ER60}
Paul Erd\H{o}s and Alfr\'ed R\'enyi.
\newblock On the evolution of random graphs.
\newblock {\em Publication of the Mathematical Institute of the Hungarian
  Academy of Sciences}, 5:17, 1960.

\bibitem{FKP89}
Philippe Flajolet, Donald~E. Knuth, and Boris Pittel.
\newblock The first cycles in an evolving graph.
\newblock {\em Discrete Mathematics}, 75(1-3):167--215, 1989.

\bibitem{FSS04}
Philippe Flajolet, Bruno Salvy, and Gilles Schaeffer.
\newblock {A}iry phenomena and analytic combinatorics of connected graphs.
\newblock {\em Electronic Journal of Combinatorics}, 11(1), 2004.

\bibitem{FS09}
Philippe Flajolet and Robert Sedgewick.
\newblock {\em Analytic Combinatorics}.
\newblock Cambridge University Press, 2009.

\bibitem{GK05}
Ira~M. Gessel and Louis~H. Kalikow.
\newblock Hypergraphs and a functional equation of {B}ouwkamp and de {B}ruijn.
\newblock {\em Journal of Combinatorial Theory, Series A}, 110(2):275--289,
  2005.

\bibitem{GZVCN09}
Gourab Ghoshal, Vinko Zlati{\'c}, Guido Caldarelli, and MEJ Newman.
\newblock Random hypergraphs and their applications.
\newblock {\em Physical Review E}, 79(6):066118, 2009.

\bibitem{H98}
Hsien-Kuei Hwang.
\newblock On convergence rates in the central limit theorems for combinatorial
  structures.
\newblock {\em European Journal of Combinatorics}, 19(3):329--343, 1998.

\bibitem{JKLP93}
Svante Janson, Donald~E. Knuth, Tomasz \L{}uczak, and Boris Pittel.
\newblock The birth of the giant component.
\newblock {\em Random Structures and Algorithms}, 4(3):233--358, 1993.

\bibitem{KL02}
M.~Karo\`nski and Tomasz \L{}uczak.
\newblock The phase transition in a random hypergraph.
\newblock {\em Journal of Computational and Applied Mathematics},
  142(1):125--135, 2002.

\bibitem{KL97}
Michal Karo\`nski and Tomasz \L{}uczak.
\newblock The number of connected sparsely edged uniform hypergraphs.
\newblock {\em Discrete Mathematics}, 171(1-3):153--167, 1997.

\bibitem{K10}
Christoph Koutschan.
\newblock A fast approach to creative telescoping.
\newblock {\em Mathematics in Computer Science}, 4(2-3):259--266, 2010.

\bibitem{Kpackage}
Christoph Koutschan.
\newblock {HolonomicFunctions} (user's guide).
\newblock (10-01), 2010.
\newblock
  http:/$\!$/www.risc.jku.at/\linebreak[0]research/\linebreak[0]combinat/\linebreak[0]software/\linebreak[0]HolonomicFunctions/.

\bibitem{MR95}
M.~Molloy and B.~Reed.
\newblock A critical point for random graphs with a given degree sequence.
\newblock {\em Random Structures and Algorithms}, 6:161--180, 1995.

\bibitem{NSW01}
M.~E.~J. Newman, S.~H. Strogatz, and D.~J. Watts.
\newblock Random graphs with arbitrary degree distributions and their
  applications.
\newblock {\em Phys. Rev. E}, 64(2):026118, July 2001.

\bibitem{O13}
B\'er\'enice Oger.
\newblock Decorated hypertrees.
\newblock {\em Journal of Combinatorial Theory, Series A}, 120(7):1871--1905,
  2013.

\bibitem{PW05}
Boris Pittel and Nicholas~C. Wormald.
\newblock Counting connected graphs inside-out.
\newblock {\em Journal of Combinatorial Theory, Series B}, 93(2):127--172,
  2005.

\bibitem{R10}
Vlady Ravelomanana.
\newblock Birth and growth of multicyclic components in random hypergraphs.
\newblock {\em Theoretical Computer Science}, 411(43):3801--3813, 2010.

\bibitem{SS85}
Jeanette Schmidt-Pruzan and Eli Shamir.
\newblock Component structure in the evolution of random hypergraphs.
\newblock {\em Combinatorica}, 5(1):81--94, 1985.

\bibitem{HS06}
Remco van~der Hofstad and Joel Spencer.
\newblock Counting connected graphs asymptotically.
\newblock {\em European Journal on Combinatorics}, 26(8):1294--1320, 2006.

\bibitem{W77}
Edward~M. Wright.
\newblock The number of connected sparsely edged graphs.
\newblock {\em Journal of Graph Theory}, 1:317--330, 1977.

\bibitem{W80}
Edward~M. Wright.
\newblock The number of connected sparsely edged graphs~{III}: Asymptotic
  results.
\newblock {\em Journal of Graph Theory}, 4(4):393--407, 1980.

\end{thebibliography}
\end{document}